\documentclass[12pt,reqno]{amsart}
\usepackage{amssymb, amsmath, amsfonts, amsthm, graphics}
\usepackage[hmargin=1 in, vmargin = 1 in]{geometry}
\usepackage{mathrsfs}
\usepackage{newpxtext}
\usepackage{newpxmath}
\usepackage{tikz-cd} 
\usepackage{enumitem}
\setlist{itemsep=0.5em}
\usetikzlibrary{matrix, calc, arrows} 
\usepackage[hyphens]{url}
\usepackage{hyperref}
\usepackage{cleveref}
\usepackage[all]{xy}
\usepackage{marginnote}

\newcommand{\isom}{\cong} 

\theoremstyle{definition}

\numberwithin{equation}{section}

\DeclareMathOperator{\Hom}{Hom}

\DeclareMathOperator{\Spec}{\text{Spec}}

\DeclareMathOperator{\Proj}{\text{Proj}}

\DeclareMathOperator\codim{codim}

\newcommand{\vol}{\mathrm{vol}}
\newcommand{\FA}{\alpha_{F}}
\newcommand{\fpt}{\text{fpt}_\fm}
\newcommand{\Ied}{I_e ^{\Delta}}

\newcommand{\CC}{\mathbb{C}}

\newcommand{\FF}{\mathbb{F}}

\newcommand{\II}{\mathbb{I}}

\newcommand{\NN}{\mathbb{N}}

\newcommand{\PP}{\mathbb{P}}
\newcommand{\QQ}{\mathbb{Q}}
\newcommand{\RR}{\mathbb{R}}

\newcommand{\ZZ}{\mathbb{Z}}

\newcommand{\cF}{\mathcal{F}}

\newcommand{\cL}{\mathcal{L}}

\newcommand{\cO}{\mathcal{O}}

\newcommand{\cS}{\mathcal{S}}

\newcommand{\fm}{\mathfrak{m}}

\newcommand{\fp}{\mathfrak{p}}

\newcommand{\Fe}{F^{e}_{*}}
\newcommand{\om}{\omega}
\newcommand{\s}{\mathscr{s}}

\theoremstyle{plain}
\newtheorem{thm}{Theorem}[section]
\newtheorem{Pn}[thm]{Proposition}
\newtheorem{Cor}[thm]{Corollary}
\newtheorem{lem}[thm]{Lemma}

\theoremstyle{definition}
\newtheorem{dfn}[thm]{Definition}

\newtheorem{eg}[thm]{Example}
\newtheorem{rem}[thm]{Remark}

\newtheorem{notation}[thm]{Notation}

\newtheorem{question}[thm]{Question}

	\title{A Frobenius version of Tian's Alpha-Invariant}
	\author{Suchitra Pande}
	\address[S.~Pande]{Department of Mathematics\\University of Utah\\Salt Lake City, 
		UT 84112\\USA}
\email{\href{mailto:suchitra.pande@utah.edu}{suchitra.pande@utah.edu}}
 \thanks{This work was partially supported by the NSF grants \#1952399, \#1801697 and \#2101075}

\begin{document}

 \begin{abstract}
       For a pair $(X,L)$ consisting of a projective variety $X$ over a perfect field of characteristic $p>0$ and an ample line bundle $L$ on $X$, we introduce and study a positive characteristic analog of Tian's $\alpha$-invariant, 
       which we call the \emph{$\FA$-invariant}. We utilize the theory of $F$-singularities in positive characteristics, and our approach is based on replacing klt singularities with the closely related notion of global $F$-regularity. We show that the $\FA$-invariant of a pair $(X,L)$ can be understood in terms of the global Frobenius splittings of the linear systems $|mL|_{m \geq 0}$. We establish inequalities relating the $\FA$-invariant with the $F$-signature, and use that to prove the positivity of the $\FA$-invariant for all globally $F$-regular projective varieties (with respect to any ample $L$ on $X$). When $X$ is a Fano variety and $L = -K_X$, we prove that the $\FA$-invariant of $X$ is always bounded above by $1/2$ and establish tighter comparisons with the $F$-signature. We also show that for toric Fano varieties, the $\FA$-invariant matches with the usual (complex) $\alpha$-invariant. Finally, we compute some examples of the $\FA$-invariant and derive consequences. 
 \end{abstract}

	\maketitle

 \section{Introduction}
 
The $\alpha$-\emph{invariant} of a complex Fano manifold $X$ was introduced by Tian in \cite{TianAlphaDefinition} to provide a sufficient criterion for the existence of a K\"ahler-Einstein metric on $X$. Though initially defined analytically, Demailly later reinterpreted the $\alpha$-invariant in terms of the \emph{log canonical threshold} \cite{CheltsovShramovDemaillyalphainvariant}, an algebraic invariant of the singularities of divisors on $X$. Since then, understanding the $\alpha$-invariant and the study of K\"ahler-Einstein metrics on Fano varieties, also called $K$-stability theory, has led to many fundamental advances in our understanding of complex Fano varieties; see \cite{OdakaSanoAlphaInvariant}, \cite{BirkarBABConjecture}, \cite{XuKstabilitysurvey}. The minimal model program (MMP), and the singularities that arise therein have played a key role in these advances.

The purpose of this paper is to study a positive characteristic analog of the $\alpha$-invariant. In our approach, we replace the singularities of the MMP with singularities defined using the Frobenius map (``\emph{$F$-singularities}"). Though $F$-singularities have fundamentally different definitions than the singularities of the MMP, a dictionary involving many precise relationships between these classes has been established; see \cite{SmithFRatImpliesRat}, \cite{HaraRatImpliesFRat}, \cite{MehtaSrinivasRatImpliesFRat}, \cite{HaraWatanabeFRegFPure}, \cite{HaraYoshidaGeneralizationOfTightClosure}, \cite{TakagiInterpretationOfMultiplierIdeals} and \cite{MaSchwedeSingularitiesMixedCharBCM}. Under this dictionary, log canonical (resp. Kawamata log-terminal (klt)) singularities correspond to $F$-split (resp. strongly $F$-regular) singularities (\Cref{defnsharpFsplitting}). Since the $\alpha$-invariant of a complex Fano variety $X$ involves the log canonicity of anti-canonical $\QQ$-divisors of $X$, this inspires our definition of the Frobenius version:
\begin{dfn} \label{introdfnalpha}
    Let $X$ be a globally $F$-regular Fano variety over a perfect field of positive characteristic.  Then, we define the $\FA$-invariant of $X$ as
    \[ \FA(X) : = \sup \{ t \geq 0 \, | \, (X, t\Delta) \, \text{is globally $F$-split } \forall \, \text{effective $\QQ$-divisor } \Delta \sim_{\QQ} -K_X   \}. \] 
\end{dfn}


\begin{rem}
    Since we intend for the $\FA$-invariant to capture global properties of anti-canonical $\QQ$-divisors on $X$, we use the notion of global $F$-splitting (\Cref{defnsharpFsplitting}), and to do so we require $X$ to be globally $F$-regular (\Cref{defn.GlobFreg}). Global $F$-splitting and $F$-regularity of a pair $(X, D)$ can be thought of as $F$-splitting and $F$-regularity of the cone over $(X, D)$ (with respect to $-K_X$) respectively. We note that in the usual definition of the $\alpha$-invariant of a Fano variety, using the klt and log canonical conditions on the cone yields the minimum value between the usual $\alpha$-invariant of $X$ and $1$ (see \Cref{comparisontoCremark}). Thus, at least for Fano varieties with $\alpha(X) \leq 1$, the $\FA$-invariant is a ``Frobenius-analog" of Tian's $\alpha$-invariant.
\end{rem}

Our first theorem proves some surprising properties of the $\FA$-invariant in contrast to the complex version, and establishes connections to the $F$-signature of $X$, another important invariant and a Frobenius version of the \emph{local volume} of singularities:
 
\begin{thm} \label{intromainthm1}
    Let $X$ be a globally $F$-regular Fano variety over a perfect field of positive characteristic. Assume that $d = \dim(X)$ is positive. Then,
    \begin{enumerate}
        \item The $\FA$-invariant of $X$ is at most 1/2 (\Cref{alphathm}).
        \item Assume that $X$ is geometrically connected over the (perfect) base field. Then, we have $\FA(X) = 1/2$ if and only if the $F$-signature of $X$ (with respect to $-K_X$) equals $\frac{\vol(-K_X)}{2^d (d+1)!}$ (\Cref{alphavsscor}). 

        \item More generally (and still assuming $X$ is geometrically connected), the $F$-signature of $X$ is at most $\frac{\vol(-K_X)}{2^d (d+1)!}$ (\Cref{alphavsscor}).

        \item When $X$ is a toric Fano variety corresponding to a fan $\cF$, then $\FA(X)$ is the same as the complex $\alpha$-invariant of $X_\CC (\cF)$, the complex toric Fano variety corresponding to $\cF$ (\Cref{Toricalphathm}).
    \end{enumerate}
\end{thm}

Part (1) of \Cref{intromainthm1} is surprising since many complex Fano varieties have $\alpha$-invariants greater than 1/2 (and less than 1). This points to a unique phenomenon in positive characteristics, and is a consequence of \emph{duality} for the Frobenius map (see \Cref{section:duality}).  Parts (1) and (4) of \Cref{intromainthm1} together recover, and provide a positive characteristic proof of the well-known fact that the $\alpha$-invariant of toric Fano varieties is at most 1/2 (see \cite[Corollary~3.6]{LiuZhuangSharpnessofTianscriterion}). A consequence of Part(2) of \Cref{intromainthm1} is a new computation of the $F$-signature for full flag varieties:

\begin{eg} (\Cref{eg:fullflagvars})
     Let $k$ be an algebraically closed field of characteristic $p>0$ and $X_n$ (for $n \geq 2$) denote the projective variety parametrizing complete flags in a fixed $n$-dimensional vector space $V$ over $k$. Set $d = \dim(X_n) = \frac{n(n-1)}{2}$. Then, we have $\FA(X_n) = \frac{1}{2}$ and consequently, by Part (2) of \Cref{intromainthm1}, the $F$-signature of $X_n$ equals $\frac{1}{d+1}$.
\end{eg}

A key input in this computation is a result on Haboush, in turn relying on the irreducibility of Steinberg representations \cite{HaboushAShortProofOfKempf}.

The $\FA$-invariant (like the complex $\alpha$-invariant) can be defined for any pair $(X,L)$, where $X$ is a globally $F$-regular projective variety and $L$ is an ample line bundle on $X$. In Sections \ref{section3} and \ref{section4}, we develop the theory of the $\FA$-invariant in this more general setting. From this perspective, the $\FA$-invariant is an asymptotic invariant of a section ring of a projective variety that shares many properties and relations with the $F$-signature. In this direction, we prove: 

\begin{thm} \label{intromainthm2}
Let $S$ denote the section ring of a globally $F$-regular projective variety over a perfect field $k$, with respect to some ample line bundle over $X$ (\Cref{sectionringdfn}). Then, 
    \begin{enumerate}
        \item $\FA (S)$ can be calculated as the following limit:
        \[ \FA(S) = \lim _{e \to \infty} \frac{m_e (S)}{p^e} \]
        where $m_e (S)$ denotes that maximum integer $m$ such that for each non-zero homogeneous element $f \in S$ of degree $m$, the map $S \to \Fe S$ sending $1$ to $\Fe f$ splits. See \Cref{finitedegreeapprox}.
        \item $\FA(S)$ is positive (\Cref{positivityofalpha}).
        \item Base-change  (\Cref{changeofbasefield}): Assume that $S_0 = k$ and $K$ is any perfect field extension of $k$. Then,
        \[\FA(S) = \FA(S \otimes_k K). \]
    \end{enumerate}
\end{thm}

Finally, we consider some examples that highlight interesting features of the $\FA$-invariant. For instance, while the $\FA$-invariant detects strong $F$-regularity of section rings, we see that the it does not detect regularity (\Cref{quadriceg}).
In \Cref{cubicsurfaceeg}, we observe that when viewed through the lens of reduction modulo $p$, the $\FA$-invariant may depend on the characteristic. Furthermore, while the limit of the $\FA$-invariant as $p \to \infty$ may exist and have an interesting geometric interpretation, the limit is not necessarily the complex $\alpha$-invariant. These examples raise interesting questions and obstructions to relating log canonical thresholds and $F$-pure thresholds of divisors on a klt variety which we hope will lead to other results in the future.

\begin{rem}
    Most results in this paper are proved in the more general, yet useful setting of globally $F$-regular log Fano pairs $(X, \Delta)$. Thus, the proofs often involve some technical perturbation arguments, especially when the coefficients of $\Delta$ have denominators divisible by $p$, the characteristic of the base field. If the reader wishes to ignore these technicalities and focus on the case of $\Delta = 0$ where the proofs are more transparent, we refer to the first version of this article available here: \url{https://arxiv.org/abs/2311.00989v1}.
\end{rem}

\paragraph{\textbf{Acknowledgements}} I would like to thank my advisor Karen Smith for her guidance, support and many helpful discussions. I would like to thank Harold Blum and Yuchen Liu for their valuable guidance that led to some of the main theorems of this paper. I am thankful for the many useful suggestions given by Devlin Mallory, Mircea Musta\c t\u a, Karl Schwede, Kevin Tucker, and Ziquan Zhuang. I would also like to thank Anna Brosowsky, Jack Jeffries, Seungsu Lee, Linquan Ma, Shravan Patankar, Claudiu Raicu, Anurag Singh, David Stapleton, Vijaylaxmi Trivedi and Yueqiao Wu for helpful conversations. I thank Devlin Mallory, Mircea Musta\c t\u a and Austyn Simpson for detailed comments on an earlier draft. Parts of this work were conducted during visits to the University of Utah, the University of Illinois at Chicago, the Tata Institute of Fundamental Research, the Indian Statistical Institute (Bangalore), and the Indian Institute of Science. I thank all these institutes for their facilities and hospitality. In addition, this material is based upon work supported by the National Science Foundation under Grant No. DMS-1928930 and by the Alfred P. Sloan Foundation under grant G-2021-16778, while the author was in residence at the Simons Laufer Mathematical Sciences Institute (formerly MSRI) in Berkeley, California, during the Spring 2024 semester.

	\section{Preliminaries}

\begin{notation}
    Throughout this paper, all rings are assumed to be Noetherian and commutative with a unit. Unless specified otherwise, $k$ will denote a perfect field of characteristic $p>0$. A \emph{variety over $k$}  is an integral (in particular, connected), separated scheme of finite type over $k$.  For a point $x$ on a scheme $X$, the residue field $\cO_{X,x}/ \fm_x$ will be denoted by $\kappa(x)$ (where $\cO_{X,x}$ is the local ring at $x$ and $\fm_x$ is the maximal ideal of the local ring). 
\end{notation} 

\begin{notation}[Divisors and Pairs]
    A \emph{prime Weil-divisor} on a scheme $X$ is a reduced and irreducible subscheme of $X$ of codimension one. A $\ZZ$-Weil divisor is a formal $\ZZ$-linear combination of prime Weil-divisors. A \emph{$\QQ$-divisor} is a formal $\QQ$-linear combination of prime Weil-divisors. By a pair $(X, \Delta)$, we mean that $X$ is a Noetherian, normal scheme and $\Delta$ is an effective $\QQ$-divisor over $X$. A \emph{projective pair} is a pair $(X, \Delta)$ where $X$ is a projective variety over $k$.
\end{notation}

\subsection{Section Rings and Modules} \label{sectionringssubsection}
 		\begin{dfn} \label{sectionringdfn}
		Let $k$ be a field and $X$ be a projective scheme over $k$. Given an ample invertible sheaf $\cL$ on $X$ and $\cF$ a coherent sheaf on $X$, the $\NN$-graded ring $S$ defined by
		$$ S =	S(X, \cL) :=  \bigoplus _{n \geq 0} H^{0}(X, \cL^{n})	$$
		is called the \emph{section ring} of $X$ with respect to $\cL$. The multiplication on $S$ is defined by the tensor-product of global sections. The affine scheme $\Spec(S)$ is called the (affine) \emph{cone over $X$} with respect to $\cL$. The \emph{section module} of $\cF$ with respect to $\cL$ is a $\ZZ$-graded $S$-module $M$ defined by
		$$M = M(X, \cL) := \bigoplus _{n \in \ZZ} H^{0}(X, \cF \otimes \cL^{n}) 	.$$
		Similarly, the sheaf corresponding to $M$ on $\Spec(S)$ is called the \emph{cone over $\cF$} with respect to $\cL$.
	\end{dfn}

\subsection{Cones over $\QQ$-divisors} \label{conesoverdivisorssubsection}
We follow the description of the cone over a $\QQ$-divisor as in \cite[Section 5]{SchwedeSmithLogFanoVsGloballyFRegular}. Let $X$ be a normal, projective variety over a field $k$ of positive dimension and $\cL$ be an ample line bundle over $X$. Let $S$ denote the section ring of $X$ with respect to $\cL$ and $Y$ denote the affine cone $\Spec(S)$. Note that the dimension of $Y$ is at least $2$ since $X$ is positive dimensional. In this situation, given any $\ZZ$-Weil divisor $D = \sum a_i D_i$ (for distinct prime Weil divisors $D_i$) on $X$, we can construct the corresponding Weil divisor $\tilde{D}$ on $Y$, the ``cone over $D$", in three equivalent ways:

\begin{enumerate}
    \item Let $\tilde{D_i}$ be the prime Weil divisor on $Y$ corresponding to the homogeneous height one prime $\fp_i \subset S$ corresponding to $D_i$. Then $\tilde{D} = \sum a_i \tilde{D_i}$.
    \item Let $\cO_X (D)$ be the reflexive sheaf on $X$ corresponding to $D$, along with a rational section $f$ that defines the divisor $D$. Then, $\tilde{D}$ is the divisor corresponding to the reflexive $S$-module defined by
    \[ M : = M(D, \cL) = \bigoplus _{j \in \ZZ} H^0 (X, \cO_X (D) \otimes \cL^j)  \]
    and the same rational function $f$ now considered as a rational section of $M$.
    
    \item Let $\fm$ denote the homogeneous maximal ideal of $S$. The scheme $Y \setminus \{\fm\}$ has a natural map $\pi: Y \setminus \{\fm\} \to X$ making it a $k^{*}$-bundle over $X$. Then, we may define $\tilde{D}$ to be the pull back of $D$ to $Y$. More precisely, near the generic point of a component of $D$, if $D$ is given by an equation $f$, then $\tilde{D}$ is defined by $\pi^{*}f$. We then take closures to obtain a Weil-divisor on $Y$. This defines a unique divisor on $Y$ since $\codim _{Y}(\fm)$ is at least $2$.
\end{enumerate}
 The construction of the cone preserves addition and linear equivalence of divisors. Thus, the cone construction extends to $\QQ$-divisors and preserves the linear equivalence of Weil-divisors. We refer to \cite[Section 5]{SchwedeSmithLogFanoVsGloballyFRegular} for the details.


\subsection{$F$-signature}
	
 Let $R$ be any ring of prime characteristic $p$. Then for any $e \geq 1$, let $F^e: R \to R$ sending $r \mapsto r^{p^e}$ be the $e^{\text{th}}$-iterate of the \emph{Frobenius morphism}.  Since $R$ has characteristic $p$, $F^e$ defines a ring homomorphism, allowing us to define a new $R$-module for each $e \geq 1$ obtained via restriction of scalars along $F^e$. We denote this new $R$-module by $F_{*} ^e R$ and its elements by $F_{*} ^e r$ (where $r$ is an element of $R$). Concretely, $F_{*} ^e R$ is the same as $R$ as an abelian group, but the $R$-module action is given by:
 $$ r\cdot F_* ^e s := F_{*} ^e (r^{p^e} s) \textrm{\quad for  $r\in R$ and $F_* ^e s \in F_* ^e R$}   .$$

 	Now let $(R, \fm)$ denote a normal local ring and $X$ denote the normal scheme $\Spec(R)$. Throughout, we will assume that $R$ is the localization of a finitely generated $k$-algebra at a maximal ideal,  which also makes it  $F$-finite (i.e., $F_{*} ^e R$ is a finitely generated $R$-module for any $e \geq 1$), with the rank of $\Fe R $ over $R$ being $p^{ed}$, where $d$ is the Krull dimension of $R$. Let $\Delta$ be an effective $\QQ$-divisor on $X = \Spec(R)$. Then, note that since $\Delta$ is effective, for any $e \geq 1$, we have a natural inclusion $R \subset R( \lceil (p^e -1) \Delta \rceil)$ of reflexive $R$-modules. Here, $R(\lceil (p^e -1) \Delta \rceil) $ denotes the $R$-module corresponding to the reflexive sheaf $\cO_X (\lceil (p^e -1) \Delta \rceil)$. Thus, applying $\Hom_R ( \, \textunderscore \, , \, R)$ to the natural inclusion $\Fe R \subset \Fe (R(\lceil (p^e -1) \Delta \rceil))
  $, we get
  \[   \Hom_R \big( \Fe R(\lceil (p^e -1) \Delta \rceil), R \big) \subset \Hom _R (\Fe R, R).  \]
Thus, given any element  $\varphi \in \Hom_R \big( \Fe R(\lceil (p^e -1) \Delta \rceil), R \big) $, it can be naturally viewed as a map $\varphi: \Fe R \to R$.
 	
 	\begin{dfn}[Splitting Ideals] \label{frrkkdfn}
 	For any $e \geq 1$, we define the subset $I_e ^{\Delta} \subseteq R$ as 
 	\[I_e ^{\Delta} = \left\{x\in R \mid \varphi(\Fe x) \in \fm \,\text{for every map } \varphi \in \Hom_R \big( \Fe R(\lceil (p^e -1) \Delta \rceil), R \big) \, \right \}.\]
 		We observe that $I_e ^{\Delta}$ is an ideal of finite colength in $R$
 	and we call 
 	 $$ a_e ^{\Delta} = \ell_{R}(R/I_{e} ^{\Delta}) $$
 	the \emph{$\Delta$-free rank} of $\Fe R$, where $\ell_R$ denotes the length as an $R$-module.
 	\end{dfn}

	\begin{dfn} \cite[Theorem 3.11, Proposition 3.5]{BlickleSchwedeTuckerFSigPairs1} \label{F-sigdfn}
 		Let $(R, \Delta)$ be a pair as above, and $a_e ^{\Delta} (R)$ denote the $\Delta$-free rank of $\Fe R$ (\Cref{frrkkdfn}). Then the \emph{$F$-signature} of $(R, \Delta)$ is defined to be the limit:
 		$$ \s(R, \Delta) := \lim_{e \to \infty} \frac{a_{e} ^{\Delta}}{p^{ed}} $$
 		where $d$ is the Krull dimension of $R$. This limit exists by \cite{BlickleSchwedeTuckerFSigPairs1}.
 	\end{dfn}
  
\subsection{$F$-signature of cones over projective varieties.}
In this subsection, we describe how we can compute the $F$-signature of cones over projective varieties (and pairs) using \emph{global} Frobenius splittings on $X$.

\begin{dfn} \label{Fsigofsectionrings}
    Let $(X, \Delta)$ be a normal, projective pair over $k$ of positive dimension and let $L$ be an ample divisor over $X$. Let $(S, \fm)$ denote the graded section ring of $X$ with respect to $L$ and let $\tilde{\Delta}$ denote the cone over $\Delta$ with respect to $L$ (\Cref{sectionringdfn}, \Cref{conesoverdivisorssubsection}). Then we define the $F$-signature of $(X, \Delta)$ with respect to $L$ as:
    \[ \s(X, \Delta;L) = \s(S_\fm, \tilde{\Delta}),\]
    where $S_\fm$ denotes the localization of $S$ at $\fm$ and the $F$-signature in the local setting is as defined in \Cref{F-sigdfn}.
\end{dfn}

The key observation that relates the local and global settings is the following lemma due to Smith:

\begin{lem} \cite[Proof of Theorem~3.10]{SmithGloballyFRegular} \label{splitting on the cone}
    Let $k$ be a field of positive characteristic $p >0$ and $X$ be a normal projective variety over $k$. Assume that $X$ is positive dimensional. Fix an ample invertible sheaf $\cL$ and let $(S, \fm)$ be the corresponding section ring (with the homogeneous maximal ideal $\fm$). Fix an effective Weil divisor $D$ over $X$ and $\tilde{D}$ be the cone over $D$ with respect to $\cL$. Then, for any $e \geq 1$, the following are equivalent:
    \begin{enumerate}
        \item The natural map $ \cO_X \to \Fe (\cO_X (D))$
    splits as a map of $\cO_X$-modules
    \item The map on the cones $ S \to \Fe (S(\tilde{D}))  $
    splits as a map of $S$-modules.
    \item The localization
     $ S_{\fm} \to \Fe (S(\tilde{D})_\fm)  $
     splits as a map of $S_\fm$-modules.
    \end{enumerate} 
\end{lem}

\begin{proof}
    The proof of (1) $\iff$ (2) is essentially contained in \cite[Proof of Theorem~3.10]{SmithGloballyFRegular} (also see \cite[Proof of Proposition 5.3]{SchwedeSmithLogFanoVsGloballyFRegular} for more details). The proof of (2) $\iff$ (3) is the same as \cite[Proposition~5.7]{DeStefaniPolstraYaoglobalFsplittingrationofmodules}.
\end{proof}

 Now fix a normal projective variety $X$ over $k$ and $\Delta$ be an effective $\QQ$-divisor over $X$. The following definition of the ``$\Ied$-subspaces" will be used extensively throughout the paper to measure ``global" Frobenius splittings. 

  \begin{dfn} \label{Iedfnwithdelta}
       For any Weil-divisor $D$ on $X$ and $e \geq 1$, define the $k$-vector subspace $I_{e} ^{\Delta}(D)$ of $H^{0}(X, \cO_X(D))$ as follows:
      $$ \Ied (D) := \{ f \in H^0 (D)\ | \ \varphi(F^e_*f )= 0 \text{ for all } \varphi \in \Hom_{\cO_X} (F^e_*\cO_X(\lceil (p^e -1) \Delta \rceil + D), \cO_X) \, \} .$$
      Here we use $H^0 (D)$ as a shorthand for the space of global sections $\Gamma(X, \cO_X(D))$.
\end{dfn}
Recall that given the effective $\QQ$-divisor $\Delta$, a map $ \varphi \in \Hom_{\cO_X} (F^e_*\cO_X(\lceil (p^e -1) \Delta \rceil + D), \cO_X)$ can be naturally thought of as an element of $ \Hom_{\cO_X} (F^e_*\cO_X(D), \cO_X)$.
\medskip

\begin{rem} \label{rem:Ieclarification}
    When we consider the case $\Delta =0$, we just write $I_e$ for $I_e ^0$. In this notation, we see that $I_e (\lceil (p^e -1) \Delta \rceil + D) \subset \Ied(D)$.
\end{rem} 

\begin{rem}
     Note that the subspace $\Ied (D)$ only depends on the sheaf $\cO_X (D)$ and not on the specific divisor $D$ in its linear equivalence class. However, the subspace $\Ied (D)$ does essentially depend on the specific divisor $\Delta$ and not just its linear equivalence class.
\end{rem}

\begin{eg} \label{eg:hypersurfaceIe}
    Let $S = k[x_1, \dots, x_n]$ be a polynomial ring over $k$ and let $G$ be a non-zero homogeneous polynomial in $S$. Let $R = S/(G)$ denote the quotient ring and assume that $R$ is normal. Then, for any $e \geq 1$, the ideal $I_e$ from \Cref{frrkkdfn} is given by the formula
    \[ I_e = (\fm ^{[p^e]} :_S (G^{p^e -1}))R\]
    where $\fm$ denotes the homogeneous maximal ideal of $S$, $\fm^{[p^e]}$ denotes the \emph{Frobenius power} of the maximal ideal defined as the ideal $(x_1 ^{p^e}, \dots,  x_n ^{p^e})$ and $:_S$ denotes the colon operation on ideals in $S$. Moreover, the subspace $I_e (mL)$ on $X := \Proj(S)$ (where $L = \cO_X(1)$) is just the degree $m$ component of the ideal $I_e$.
\end{eg}

\begin{lem} \label{lem:addingeffective}
Suppose $D$ and $D'$ are effective Weil divisors such that $D - D' = E$ is effective. Then for any $e \geq 1$, the defining map $\cO_X \to \cO_X (E)$ induces an inclusion $\Ied(D') \subset \Ied(D)$. Similarly, if $\Delta' \geq \Delta$ are two effective $\QQ$-divisors, then we have an induced inclusion
$ I_e ^{\Delta}  \subset I_e ^{\Delta'}.$
\end{lem}
\begin{proof}
    See \cite[Lemma~4.12]{LeePandeFsignaturefunction}.
\end{proof}

\begin{lem} \label{lem:coefficientmorethanone}
    Suppose $\Delta \geq 0$ is a $\QQ$-divisor on $X$ such that the coefficient of $\Delta$ along some component $E$ is larger than one. Then $\Ied(D) = H^0 (X, \cO_X(D))$ for any divisor $D$. 
\end{lem}
\begin{proof}
    Since the coefficient along $E$ of $\Delta$ is larger than one, we have $\lceil (p^e -1 ) \rceil \geq p^e E$. Therefore, the claim follows from noting that the map $\cO_X \to \Fe \cO_X (p^e E) $ can not be split for any non-zero effective divisor $E$.
    \end{proof}
\begin{rem} \label{globalIevslocal}
    Let $L$ be an ample divisor on $X$. Then, it follows from \Cref{splitting on the cone} that $\Ied (mL)$ is the degree $m$ component of the $\Delta$-splitting ideal of the localization of the section ring of $S_\fm$ with respect to $L$ (\Cref{frrkkdfn}).
\end{rem}

\begin{lem} \cite[Lemma~4.7]{LeePandeFsignaturefunction} \label{formulaforFsigwithdelta}
    Let $L$ be an ample Cartier divisor on $X$ and $(S, \fm)$ denote the section ring of $X$ with respect to $L$. Let $\Delta_{S}$ denote the cone over $\Delta$ with respect to $L$ (\Cref{conesoverdivisorssubsection}). Then, for any $e \geq 1$, if $a_{e} ^{\Delta} (L)$ denotes the $\Delta_{S_\fm}$-free-rank of $\Fe S_\fm$ (\Cref{frrkkdfn}), then $a_{e} ^{\Delta} (L)$ is computed by the following formula:
\begin{equation} \label{freerankformula}
    a_{e} ^{\Delta} (L) = \frac{1}{[k':k]} \, \sum _{m = 0} ^{\infty} \dim_{k} \frac{H^{0}(X, mL)}{I_{e}  ^{\Delta} (mL)}
\end{equation}
where $k'$ denotes the field $H^0 (X, \cO_X)$.
Hence, the $F$-signature of $(X, \Delta)$ with respect to $L$ can be computed as
$$ \s(X, \Delta; L)  =   \frac{1}{[k':k]} \, \lim _{e \to \infty }  \frac{  \sum \limits _{m = 0} ^{\infty} \dim_{k} \frac{H^{0}(X, mL)}{I_{e} ^{\Delta}(mL)}}{p^{e(\dim(X)+1)}}$$
\end{lem}
\begin{proof}
    Recall that by \Cref{Fsigofsectionrings}, we have $\s(X, \Delta;L) = \s(S_\fm, \Delta_S)$. By \Cref{globalIevslocal}, we know that 
    \[ I_{e} ^{\Delta_S} = \bigoplus _{m \geq 0} I_e ^{\Delta} (mL).    \]
    See \Cref{frrkkdfn} for the definition of $I_e ^{\Delta_S}$.
    Therefore, we have
    \[ \ell_{S} (S/I_{e} ^{\Delta_S}) = \sum _{m \geq 0} \dim_{k'} \frac{H^0 (mL)}{I_{e} ^{\Delta} (mL)} = \frac{1}{[k':k]} \, \sum _{m = 0} ^{\infty} \dim_{k} \frac{H^{0}(X, mL)}{I_{e}  ^{\Delta} (mL)} \]
    where $\ell_S$ denotes the length as an $S$-module. This completes the proof of the lemma.
\end{proof}

\subsection{Duality and the Trace Map} \label{section:duality} We continue to work over any perfect field $k$ of characteristic $p>0$. But in this subsection, we assume that $H^0 (X, \cO_X) = k$, i.e., that $X$ is geometrically connected. It will be convenient to think of the subspaces $I_{e} ^{\Delta}$ (\Cref{Iedfnwithdelta}) using a pairing arising out of duality for the Frobenius map. Let $ (X, \Delta) $ be a projective pair and $D$ be any Weil divisor on $X$. 

Recall that by applying duality to the Frobenius map, we get the following isomorphism of reflexive $\cO_X$-modules:
\begin{equation} \label{dualityiso}
\mathscr{H}om _{\cO_X}(F^e_*\cO_X(\lceil (p^e -1) \Delta \rceil + D), \cO_X) \cong F^e_*\cO_X( -(p^e-1 )K_X - \lceil (p^ e -1 ) \Delta \rceil - D ). \end{equation}
Here, since $X$ is a normal variety, we can define the canonical divisor $K_X$ by extending the canonical bundle from the smooth locus. See \cite[Section 4.1]{SchwedeSmithLogFanoVsGloballyFRegular} for a detailed discussion regarding duality for the Frobenius map. 
Furthermore, when $D = 0$, this gives an isomorphism
\begin{equation} \label{dualtohomdelta}  \mathscr{H}om _{\cO_X}(F^e_*\cO_X (\lceil (p^e -1 )\Delta \rceil , \cO_X) \cong F^e_*\cO_X( -(p^e-1 )K_X - \lceil (p^e -1) \Delta \rceil ).\end{equation}
Composing this isomorphism with the evaluation at $\Fe 1$ map, we obtain the $\Delta$-trace map:
\[ \text{tr}^\Delta _e : \Fe \big(\cO_{X}( (1 - p^e)K_X - \lceil (p^e -1) \Delta \rceil)\big)  \to \cO_{X} \]
which is a non-zero map of $\cO_X$-modules. When $\Delta = 0$, we obtain the trace map:
$\text{tr} _e : \Fe (\cO_X((1-p^e)K_X)) \to  \cO_X $.
Moreover, by definition, it is clear that the map $\text{tr} ^\Delta _e$ is just the restriction of $\text{tr} _e$ to the subspace $ \Fe \big( \cO_{X}( (1 - p^e)K_X - \lceil (p^e -1) \Delta \rceil) \big)$ via the inclusion
\[ \Fe \big(\cO_X((1-p^e)K_X - \lceil(p^e -1) \Delta \rceil ) \big) \to \Fe \big(\cO_X((1-p^e)K_X ) \big) \]
induced by the effective divisor $\Delta$. By construction, the trace map $\text{tr} _e$ is compatible with localization and in particular, given a smooth point $x \in X
$, the stalk of the map $\text{tr} _e$ at $x$ can be identified with the generating map of the $\Fe R$-module $\Hom_R (\Fe R , R)$, where the $R$ is the local ring at $x$. In general, the maps $\text{tr}_e$ encode the the singularities of the local rings of $X$.

Next, we may consider the map induced by $\text{tr}_e$ on the global sections to obtain the global $\Delta$-trace map
\begin{equation} \label{tracemapfordelta} \text{Tr} _e ^\Delta  : H^0 \Big(X, \Fe \big(\cO_{X}( (1 - p^e)K_X - \lceil (p^e -1) \Delta \rceil)\big) \Big) \to k = H^0(X, \cO_{X}). \end{equation}
This global trace map encodes the subtle properties of global Frobenius splittings on $X$. Note that even though $\text{tr}_e$ is always a non-zero map of $\cO_X$-modules, the map on the global sections $\text{Tr}_e$ can be the zero map. 


\begin{lem} \label{tracemaplemma}
    The kernel of the trace map $\text{Tr}_e ^{\Delta}$ in \Cref{tracemapfordelta} is exactly the subspace $\Fe I_{e} ^{\Delta} ((1 - p^e) K_X - \lceil (p^e -1) \Delta \rceil )$. See \Cref{Iedfnwithdelta} for the definition of the subspace $\Ied$. 
\end{lem}

\begin{proof} By the definition of $\Ied$, a section $f \in H^0 (X, \cO_X ((1 - p^e) K_X - \lceil (p^e -1) \Delta \rceil ))$ is contained in the corresponding $\Ied$-subspace if and only if for every
    \[ \varphi \in \Hom _{\cO_X} ( \Fe (\cO_X((1 - p^e) K_X  )), \cO_X ) \isom H^0 (X, \Fe \cO_X) \isom \Fe k, \]
    we have $\varphi(\Fe f) = 0$. But, $\Hom  _{\cO_X} ( \Fe (\cO_X((1 - p^e) K_X  )), \cO_X )$ is a one dimensional $k$-vector space, and is generated by the trace map $\text{tr} _e$. Therefore, $f$ is contained in the corresponding $\Ied$-subspace, if and only if when viewed via the inclusion
    \[ \Fe \cO_X ((1-p^e)K_X - \lceil (p^e -1) \Delta \rceil ) \to \Fe \cO_X((1-p^e)K_X), \]
    we have $\text{Tr}_e (\Fe f) =0$.
    Since the map $\text{Tr}_{e} ^{\Delta}$ is exactly the restriction of the trace map $\text{Tr} _{e}$ to the subspace  $H^0 \Big(X, \Fe \big( \cO_{X}( (1 - p^e)K_X - \lceil (p^e -1) \Delta \rceil) \big) \Big)$, the lemma follows.
\end{proof}

\begin{lem} \label{splittingviatrace}
    Let $X$ be a normal projective variety over $k$ and $D$ be a Weil divisor over $X$. Then, a section $f \in H^0 (X, \cO_X(D))$ is contained in $I_e (D)$ (\Cref{Iedfnwithdelta}) if and only if for all sections $g \in H^0 (X, \cO_X((1-p^e)K_X - D))$, we have $\text{Tr} _e (fg)  = 0$.
\end{lem}
\begin{proof}
    Let $U \subset X$ denote the regular locus of $X$. Since $X$ is normal, the closed set $X \setminus U$ has codimension at least $2$ in $X$. Let $D_U$ and $K_U$ denote the restrictions of $D$ and $K_X$ respectively (so that $K_U$ is the canonical divisor on $U$). Since $U$ is a regular scheme, the iterated Frobenius map $F^e: U \to U $ is a finite flat map of schemes. Then, following \cite[III Ex. 6.10]{Hartshorne}, the isomorphim
    in \Cref{dualityiso} is constructed as follows: there is a natural map
    \[ \varphi:  \mathscr{H}om _{\Fe \cO_U} (\Fe (\cO_U(D_U)), \mathscr{H}om _{\cO_U} (\Fe \cO_U , \cO_U) )  \to \mathscr{H}om_{\cO_U} (\Fe ( \cO_U(D)), \cO_U) \]
    of $\Fe \cO_U$-modules that sends a map $\psi$ on the left-hand side to the map defined by 
    \begin{equation} \label{formulaforduality} \varphi(\psi): \Fe f \mapsto \text{eval at $\Fe 1$ of} \, \big(\psi (\Fe f) \big).   \end{equation}
    Moreover, $\varphi$ is an isomorphism. Now, the isomorphism in \Cref{dualityiso} is obtained by identifying $\mathscr{H}om_{\cO_U} (\Fe \cO_U, \cO_U)$ with the sheaf $\Fe (\cO_U((1-p^e)K_U))$, so that 
    \[  \mathscr{H}om _{\Fe \cO_U} (\Fe (\cO_U(D_U)), \mathscr{H}om _{\cO_U} (\Fe \cO_U , \cO_U) ) \isom \Fe (\cO_U ((1-p^e)K_U - D_U)).  \]
    Therefore, if $\Fe f \in H^0 (U, \Fe \cO_U (D_U))$ is a section, and $\Fe g \in H^0 (U, \Fe \cO_U ((1-p^e)K_U - D_U))$ which, via the above isomorphism, we view as a map $\psi_g : \Fe (\cO_U (D_U)) \to  \cO_U $, then by \Cref{formulaforduality} we see that
    \[ \psi_g (\Fe f) = \text{tr}_e (\Fe (fg)). \]
    Now the lemma follows by the definition of $I_e$ and considering the fact that global sections over $X$ and $U$ are the same since all sheaves involved are reflexive.\end{proof}

\begin{lem} \label{dualIe}
    Let $(X, \Delta) $ be a normal projective pair, and $D$ be any Weil divisor on $X$. Then, denoting $D_1 = (1-p^e) K_X - \lceil(p^e -1) \Delta \rceil - D$ and $D_2 = (1-p^e)K_X - \lceil (p^e -1) \Delta \rceil $ for any $e \geq 1$,  we have a non-degenerate pairing
    $$\frac{ H^{0} ( D)} { \Ied (D)} \times \frac{H^0 (D_1) } {\Ied (D_1)} \to \frac{H^{0}(D_2) }{\Ied ( D_2)} $$
    obtained from multiplication (and reflexifying) global sections.
In particular, $$ \dim_{k} \frac{H^{0} ( D) } {\Ied (D)} = \dim_{k} \frac{H^0 (D_1) }  {\Ied (D_1)}. $$
    
\end{lem}

\begin{proof}Using \Cref{dualityiso} and as in the previous lemma, the natural multiplication map
    $$ H^0 (D) \times H^0 (D_1) \to H^0 ( D_2)$$
    can be identified with the evaluation map
    $$ H^0 \Big(\Fe \big(\cO_X(D) \big) \Big) \times \Hom_{\cO_X} \Big (\Fe \big(\cO_{X} ( \lceil (p^e -1) \Delta \rceil + D)\big) , \cO_{X} \Big) \to \Hom_{\cO_X} (\Fe \cO_{X}(\lceil (p^e -1) \Delta \rceil), \cO_{X}) $$
    where we identify $ \Hom_{\cO_X} (\Fe \cO_{X} ( \lceil (p^e -1) \Delta \rceil + D) , \cO_{X})$ and $\Hom_{\cO_X} (\Fe \cO_{X}(\lceil (p^e -1) \Delta \rceil), \cO_{X})$ as  subspaces of $ \Hom_{\cO_X} (\Fe \cO_{X} ( D) , \cO_{X})$ and $ \Hom_{\cO_X} (\Fe \cO_{X}, \cO_{X})$ respectively, both
    via the natural inclusion $  \cO_X \to  
    \cO_X (\lceil (p^e -1) \Delta \rceil)$.
    Therefore, by \Cref{splittingviatrace} and the fact that $\text{Tr} _e ^\Delta$ is just the restriction of $\text{Tr}_e$, a section $f \in H^0 (D)$ is contained in $\Ied (D)$ if and only if for all sections $g \in H^{0}(D_1)$, the multiplication $f \, g$ satisfies $\text{Tr}_e ^\Delta (\Fe(fg)) = 0$. By \Cref{tracemaplemma}, this is equivalent to $fg$ being contained in $\Ied (D_2)$. By symmetry, a section $g \in H^0 (D_1)$ is contained in $ \Ied (D_1)$ if and only if for all sections $f \in H^0 (D)$,  we have $ g \, f \in \Ied(D_2)$. This proves that there is a well defined, and non-degenerate pairing as needed.

    Finally, we note that since by \Cref{tracemaplemma} again, $\Ied(D_2)$ is the kernel of the trace map (\Cref{tracemapfordelta}), the vector space $H^0 (D_2) /\Ied (D_2)$ is either one-dimensional over $k$, or equal to $0$. In either case, the equality of dimensions follows.
\end{proof}

\subsection{$F$-splitting and $F$-regularity} Now we collect the main definitions and properties of some $F$-singularities that we will use throughout the paper.

\begin{dfn}[Sharp $F$-splitting] \cite[Definition 3.1]{SchwedeSmithLogFanoVsGloballyFRegular} \label{defnsharpFsplitting}  Let $X$ be a normal variety over $k$ and $\Delta \geq 0$ be an effective $\QQ$-divisor. The pair $(X, \Delta) $ is said to be \emph{globally sharply $F$-split} (resp. \emph{locally} sharply $F$-split) if there exists an integer $e > 0$, such that, the natural map 
\begin{equation*} 
\cO_X \to F^e_* \cO_X( \lceil (p^e -1) \Delta \rceil)
 \end{equation*}
splits (resp. splits locally) as a map of $\cO_{X}$-modules.
A normal variety $X$ is said to globally $F$-split if the pair $(X, 0)$ is globally sharply $F$-split. 
\end{dfn}

\begin{dfn}[$F$-regularity]\label{defn.GlobFreg} \cite[Definition 3.1]{SchwedeSmithLogFanoVsGloballyFRegular} Let $X$ be a normal variety over $k$ and $\Delta \geq 0$ be an effective $\QQ$-divisor. The pair $(X, \Delta) $ is said to be \emph{globally $F$-regular} (resp. locally strongly $F$-regular) if for any effective Weil divisor $D$ on $X$, there exists an integer $e \gg 0$, such that, the natural map 
\begin{equation*} 
\cO_X \to F^e_* \cO_X( \lceil (p^e -1) \Delta \rceil + D)
 \end{equation*}
splits (resp. splits locally) as a map of $\cO_{X}$-modules.
A normal variety $X$ is said to globally $F$-regular if the pair $(X, 0)$ is globally $F$-regular. Similarly, a ring $R$ is called strongly $F$-regular if the pair $(\Spec(R), 0)$ is locally (equivalently, globally) strongly $F$-regular.

\end{dfn}

\begin{rem} When $X= \Spec(R)$ is an affine variety and $\Delta$ is an effective $\QQ$-divisor, the pair $(X, \Delta)$ being globally $F$-regular (resp. globally sharply $F$-split) is equivalent to the pair $(R, \Delta)$ being locally strongly $F$-regular (resp. locally sharply $F$-split) \cite{SchwedeSmithLogFanoVsGloballyFRegular}.
\end{rem}

\begin{rem}\label{FregularityremarK}
A local domain $R$ is strongly $F$-regular if and only if its $F$-signature $\s(R)$ is positive \cite{AberbachLeuschke}. More generally, a pair $(R, \Delta)$ (where $R$ is normal, local) is strongly $F$-regular if and only if the $F$-signature $\s(R, \Delta)$ (\Cref{F-sigdfn}) is positive \cite[Theorem 3.18]{BlickleSchwedeTuckerFSigPairs1}.
\end{rem}

The following basic properties of sharp $F$-splitting and $F$-regularity established in \cite{SchwedeSmithLogFanoVsGloballyFRegular} will be used repeatedly, so we record them here for the convenience of the reader:
\begin{lem} \label{lemdecreasingcoefficients}
    Let $X$ be a normal variety and $\Delta$ and  $\Delta'$ be effective $\QQ$-divisors on $X$. 
    \begin{enumerate}
        \item Suppose $\Delta' \leq \Delta$. Then, if $(X, \Delta)$ is globally $F$-regular (resp. globally sharply $F$-split) then $(X, \Delta')$ is globally $F$-regular (resp. globally sharply $F$-split). See \cite[Lemma~3.5]{SchwedeSmithLogFanoVsGloballyFRegular}.
        \item Suppose $(X, \Delta)$ is globally $F$-regular and $(X, \Delta+ \Delta')$ is globally sharply $F$-split, then $(X, \Delta + \lambda \Delta')$ is globally $F$-regular for every rational number $0 <\lambda<1$. See \cite[Lemma~4.9 (iii)]{SchwedeSmithLogFanoVsGloballyFRegular}.
    \end{enumerate}
\end{lem}

\begin{thm} \cite[Theorem 3.10]{SmithGloballyFRegular} 
    Let $X$ be a projective variety over $k$. Then, $X$ is globally $F$-regular if and only if the section ring $S(X, L)$ (\Cref{sectionringdfn}) with respect to some (equivalently, every) ample invertible sheaf $L$  is strongly $F$-regular.
\end{thm}

\begin{rem} (Locally) Strongly $F$-regular varieties are normal and Cohen-Macaulay. Similarly, globally $F$-regular varieties enjoy many nice properties such as: 
\begin{itemize}
    \item As proved in \cite[Theorem 4.3]{SchwedeSmithLogFanoVsGloballyFRegular}, they are \emph{log Fano type}. More precisely, there exists an effective divisor $\Delta \geq 0$ such that the pair $(X,\Delta)$ is globally $F$-regular and $-K_X - \Delta$ is ample.
    \item A version of the Kawamata-Viehweg vanishing theorem holds on all globally $F$-regular varieties \cite[Theorem 6.8]{SchwedeSmithLogFanoVsGloballyFRegular}.
\end{itemize} 
\end{rem}

\begin{thm}[\cite{SmithGloballyFRegular}, Corollary 4.3] \label{vanishingforGFR}
Let $X$ be a projective, globally $F$-regular variety over $k$. Suppose $L$ is a nef invertible sheaf over $X$. Then,
$$  H^{i}(X, L) = 0 \quad \text{for all $i >0$}.   $$
\end{thm}

\subsection{Perturbation, transformation rule and orbifold cones.}
We first prove a result that allows us to perturb by any Weil-divisor while computing the $F$-signature.

\begin{Pn}  \label{twistedFsig}
Let $(X, \Delta)$ be a normal projective pair over $k$ and $L$ an ample divisor on $X$. Assume $H^0 (X, \cO_X) = k$. Fix a (not necessarily effective) Weil divisor $D$ on $X$. Then, there exists a constant $C >0$ (depending only on $D$, $L$ and the pair $(X, \Delta)$) such that
$$ \left| \dim_{k}\frac{H^{0}(mL)}{I_{e} ^\Delta(mL)} - \dim_{k}\frac{H^{0}(mL+D)}{I_{e} ^\Delta(mL+D)} \right| \leq C p^{e(\dim(X) - 1)} $$
for all $m > 0$ and $e>0$.
\end{Pn}

In the proof of this proposition, we will need to restrict a Weil divisor on a normal variety (that is not necessarily Cartier) to a normal, complete intersection subscheme. There is a natural way to do this which we will explain in \Cref{restrictingWeildivisors}.

\begin{proof}[Proof of \Cref{twistedFsig}]
First we prove the case when $D$ is effective: For any $m \geq 1$, by using the natural map $\cO_{X}(mL) \to \cO_{X}(mL+D)$, we will view $H^{0}(mL)$ as a subspace of $H^{0}(mL+D)$. Let $J_{e}^\Delta(mL)$ denote the subspace $H^{0}(mL) \cap I_{e} ^\Delta (mL+D)$. By \Cref{lem:addingeffective}, we see that $I_{e}^\Delta (mL) \subset J_{e} ^\Delta(mL)$. Moreover, by Equation~4.12 in
\cite[Proof of Lemma~4.14]{LeePandeFsignaturefunction}, we have
\[ \dim_{k}\frac{H^{0}(mL+D)}{I_{e} ^\Delta(mL+D)} = \dim_{k}\frac{H^{0}(mL)}{J_{e} ^\Delta (mL)} + \dim_{k}\frac{H^{0}(mL+D)}{H^0 (mL) + I_{e} ^\Delta (mL+D)}.\]
Using this and the triangle inequality, we obtain that
\begin{equation} \label{firstestimate} \left| \dim_{k}\frac{H^{0}(mL)}{I_{e} ^\Delta (mL)} - \dim_{k}\frac{H^{0}(mL+D)}{I_{e} ^\Delta (mL+D)} \right| \leq \dim_{k} \frac{H^{0}(mL+D)}{H^{0}(mL)} + \dim_{k}\frac{J_{e} ^\Delta(mL)}{I_{e}^\Delta(mL)} \end{equation}
for all $m, e>0$. Next, to compute the second term in the above inequality, fix an $e>0$ and set $\Delta_e = \Delta+ \frac{1}{p^e -1} D$. Then, we observe that the subspace $J_e ^\Delta (mL)$ is exactly the same as $I_e ^{\Delta_e} (mL)$ (\Cref{Iedfnwithdelta}). Moreover, denoting $D'_e = D + \lceil (p^e -1 ) \Delta \rceil $, we also similarly have 
\begin{equation} \label{dualJe} I_e ^{\Delta_e} ((1-p^e)K_X - mL-D'_e) = H^0 ((1-p^e)K_X - mL - D_e ' ) \cap I_e ^{\Delta} ((1-p^e)K_X - mL - \lceil (p^e -1 ) \Delta \rceil) .\end{equation}
Thus, by \Cref{dualIe}, we have
\begin{equation} \label{perturbationlemfirstdual} \dim_k \frac{H^0 (mL)}{I_e ^\Delta (mL)}  = \dim_k \frac{H^0 ((1-p^e)K_X - mL - \lceil (p^e -1) \Delta \rceil)}{I_e ^\Delta ((1-p^e)K_X - mL - \lceil (p^e -1) \Delta \rceil)} \end{equation}
and similarly,
\begin{equation} \label{perturbationlemseconddual} \dim_k \frac{H^0 (mL)}{J_e ^\Delta (mL)}  = \dim_k \frac{H^0 ((1-p^e)K_X - mL -D'_e)}{I_e ^{\Delta_e} ((1-p^e)K_X - mL-D'_e)} .\end{equation}
By \Cref{dualJe}, we see 
the natural map from $H^0 ((1-p^e)K_X - mL -D' _e)$ to $ H^0 ((1-p^e)K_X - mL - \lceil (p^e -1) \Delta \rceil )$ given by adding $D_e$ restricts to an \emph{injective} map
\[ \frac{H^0 ((1-p^e)K_X - mL -D'_e)}{I_e ^{\Delta_e} ((1-p^e)K_X - mL-D'_e)} \hookrightarrow \frac{H^0 ((1-p^e)K_X - mL - \lceil(p^e -1) \Delta \rceil )}{I_e ^\Delta ((1-p^e)K_X - mL - \lceil (p^e -1)  \Delta \rceil)} .\]
By \Cref{perturbationlemfirstdual} and \Cref{perturbationlemseconddual}, the dimension of the cokernel of this map is exactly equal to $  \dim_{k}\frac{J_{e} ^\Delta (mL)}{I_{e} ^\Delta (mL)} = \dim_{k}\frac{H^{0}(mL)}{I_{e} ^\Delta (mL)} - \dim_{k}\frac{H^{0}(mL)}{J_{e} ^\Delta (mL)}$, which we may now estimate as
\begin{equation} \label{secondestimate}
   \dim_{k}\frac{J_{e} ^\Delta(mL)}{I_{e} ^\Delta (mL)} \leq  \dim_{k} \frac{H^{0}((1-p^e)K_{X} -mL - \lceil (p^e -1) \Delta \rceil)}{H^{0}((1-p^e)K_{X} -mL-D - \lceil (p^e -1) \Delta \rceil)}.
\end{equation}

Pick an integer $M \gg 0$ such that following conditions hold for the linear system $|ML|$: 
\begin{itemize}
    \item There is an effective divisor $E_1 \in |ML|$ such that $E_1 \geq D$.
    \item There is an effective  divisor $E_2 \in |ML|$ and an effective Weil-divisor $E_3$ such that $E_2 - E_3 \sim -K_X$. We can do this by choosing $M$ such that $H^0 (X,\cO_X(K_X+ ML)) \neq 0$.
    \item There is an effective divisor $E \in |ML|$ such that $E$ is a reduced, irreducible and normal subvariety of $X$ such that $E$ does not occur as a component of the support of either $E_2, E_3$ or $\Delta$.
    \item $mL$ admits a global section that doesn't vanish along $E$ for all $m \geq M$.
\end{itemize}
This is possible since $L$ is ample. Now, we firstly note that since $E_1 \geq D$, we have $ \dim_k H^{0}((1-p^e)K_{X} -mL-E_1 - \lceil (p^e -1) \Delta \rceil) \leq \dim_k H^{0}((1-p^e)K_{X} -mL-D - \lceil (p^e -1) \Delta \rceil)$. And moreover, since the dimension of the space of global sections  only depends on the linear equivalence of divisors, by switching $E$ for $E_1$, we have
$ \dim_{k} \frac{H^{0}((1-p^e)K_{X} -mL - \lceil (p^e -1) \Delta \rceil)}{H^{0}((1-p^e)K_{X} -mL-D - \lceil (p^e -1) \Delta \rceil)} $ is at most $ \dim_{k} \frac{H^{0}((1-p^e)K_{X} -mL - \lceil (p^e -1) \Delta \rceil)}{H^{0}((1-p^e)K_{X} -mL-E - \lceil (p^e -1) \Delta \rceil)} $. Now, since $E$ is assumed to be normal and does not share any components with $E_2, E_3 $ or $\Delta$, for any $e \geq 1$, we may restrict the Weil divisor $(1-p^e) K_X - mL - \lceil(p^e -1) \Delta \rceil $ to $E$ (as explained in \Cref{restrictingWeildivisors}). In particular, using \Cref{restrictingtonormaldivisorexactsequence} and \Cref{secondestimate}, we get
\[ \dim_{k}\frac{J_{e} ^\Delta (mL)}{I_{e} ^\Delta (mL)} \leq \dim_{k} H^{0}((1-p^e)K_{X}|_E -mL|_E - (\lceil (p^e -1) \Delta \rceil)|_E).  \]
Furthermore, since by assumption, we have $-K_X \sim E_2 - E_3 $ and $E_2$, $E_3$ and $\Delta$ do not share any components with $E$ and since $mL$ also admits a section that does not share a component with $E$, we see that $mL|_E + \lceil (p^e -1) \Delta \rceil|_E + (p^e - 1) E_3 |_E$ is effective. Hence, using that $E_2 \sim -K_X + E_3$, we have:
\begin{equation} \label{thirdestimate}
    \dim_{k} \frac{J_{e} ^\Delta (mL)}{I_{e} ^\Delta (mL)} \leq \dim_k H^0 ((p^e -1) E_2 |_E)  = \frac{\vol(E_2|_E)}{(d-1)!} p^{e(d-1)} + o(p^{e(d-2)})
\end{equation}
 for $e \gg 0$. Since we have $E_1 \geq D$ and $E_1$ is Cartier, by considering the exact sequence
 \[ 0 \to \cO_X(-E_1) \to \cO_X \to \cO_{E_1} \to 0 \]
and using the left exactness of $H^0 (X, \_)$, we also have
 \begin{equation} \label{fourthestimate}
     \dim_k \frac{H^{0}(mL+D)}{H^{0}(mL)} \leq  \dim_{k} \frac{H^{0}(mL+E_1)}{H^{0}(mL)}  \leq \frac{\vol(L|_{E_1})}{(d-1)!} m^{d-1} + o(m^{d-2}).
 \end{equation}
 for all $m \gg 0$.
 
 Finally, by \cite[Theorem 4.9]{LeePandeFsignaturefunction}, we may pick a constant $C_{2}>0$ such that $H^{0}(mL) = I_{e} ^\Delta (mL)$ and $H^{0}(mL+D) = I_{e} ^\Delta (mL+D)$ for $m >C_{2}p^{e}$ (note that since $\Delta$ is effective, we have $I_e \subset \Ied$ by \Cref{lem:addingeffective}). Therefore, to prove the Proposition, it is enough to consider the case when $m \leq C_{2} p^{e}$. In this case, the Proposition now follows by putting together inequalities in \ref{firstestimate}, \ref{thirdestimate} and \ref{fourthestimate}. This completes the proof of the Proposition when $D$ is effective.

More generally, we first pick an $r \gg 0$ such that $rL$ and $D+ rL$ are both linearly equivalent to effective divisors. Then, for any $ e \geq 1$ and $m > r$, we have
\begin{equation*}
\begin{split}
 \left| \dim_{k}\frac{H^{0}(mL)}{I_{e} ^\Delta(mL)} - \dim_{k}\frac{H^{0}(mL+D)}{I_{e} ^\Delta(mL+D)} \right| \leq & \left| \dim_{k}\frac{H^{0}(mL)}{I_{e} ^\Delta(mL)} - \dim_{k}\frac{H^{0}((m-r)L)}{I_{e} ^\Delta ((m-r)L)} \right| \\
 + &\left| \dim_{k}\frac{H^{0}((m-r)L)}{I_{e} ^\Delta ((m-r)L)} - \dim_{k}\frac{H^{0}((m-r)L + rL+D)}{I_{e} ^\Delta ((m-r)L+rL+D)} \right| .
 \end{split}   \end{equation*}
 Now, we may apply the previous case of the Proposition (since both $D+rL$ and $rL$ are effective) to each of the two terms in the above inequality. Since $r$ was independent of $e$, this completes the proof of the Proposition .
\end{proof}

\begin{dfn}[$F$-signature of certain orbifold cones] \label{def:Fsigoforbifoldcones}
    Let $(X,\Delta)$ be a projective, globally $F$-regular pair over $k$ of positive dimension. Let $L$ be a $\ZZ$-Weil divisor that is also $\QQ$-Cartier and ample (meaning that $rL$ is Cartier and ample for some integer $r >0$). Then, we may consider the orbifold cone of $(X, \Delta)$ with respect to $L$ by considering the section ring $S(X, L) := \bigoplus_{m \geq 0} H^0 (X, \cO_X(mL))$, and the pull-back $\tilde{\Delta}$ of $\Delta$ under the natural map $\pi: Y \setminus \{0\} \to X$, where $Y = \Spec(S)$ and $0$ is the point corresponding to the homogeneous maximal ideal of $S$. Then, we may
    we define the $F$-signature of $(X, \Delta)$ with respect to $L$ as the $F$-signature
    \[ \s(X, \Delta;L) := \s(Y, \tilde{\Delta}, 0) . \]
    By the exact same argument as \Cref{formulaforFsigwithdelta}, this $F$-signature may be computed by the formula
    \[ \s(X, \Delta; L)  =   \frac{1}{[k':k]} \, \lim _{e \to \infty }  \frac{  \sum \limits _{m = 0} ^{\infty} \dim_{k} \frac{H^{0}(X, mL)}{I_{e} ^{\Delta}(mL)}}{p^{e(\dim(X)+1)}} . \]
\end{dfn}

\begin{thm}  \label{transformationrulewithpairs}
    Let $(X,\Delta)$ be a projective, globally $F$-regular pair over $k$. Let $L$ be a $\ZZ$-Weil divisor that is also $\QQ$-Cartier and ample.  Let $n \geq 1$ be any integer. Then, we have the transformation rule
    \[ \s(X, \Delta; L) = n \, \s(X, \Delta; L^n). \]
\end{thm}

\begin{proof}
    Let $r >0$ be an integer such that $rL$ is Cartier and ample. Then, it clearly suffices to prove the theorem when $r$ divides $n$, which we assume for the rest of the proof.
    
    Set $k' = H^0 (X, \cO_X)$ and $\ell = [k' :k]$. By \Cref{formulaforFsigwithdelta} and \Cref{def:Fsigoforbifoldcones}, we have the formulas
    \[  a_{e} ^{\Delta} (L) = \frac{1}{\ell} \, \sum _{m = 0} ^{\infty} \dim_{k} \frac{H^{0}(X, mL)}{I_{e}  ^{\Delta} (mL)} \]
    and similarly
    \[  a_{e} ^{\Delta} (L^n) = \frac{1}{\ell} \, \sum _{m = 0} ^{\infty} \dim_{k} \frac{H^{0}(X, mnL)}{I_{e}  ^{\Delta} (mnL)}. \]
    Now, applying \Cref{twistedFsig} to the divisors $L, 2L, \dots, (n-1)L$, we may pick a constant $C> 0$ such that for each $1 \leq i \leq n-1$ we have
    \begin{equation} \label{perturbationfortransofrmationrule} \left| \dim_{k}\frac{H^{0}(mnL)}{I_{e} ^\Delta(mnL)} - \dim_{k}\frac{H^{0}((i + mn)L)}{I_{e} ^\Delta((i+mn)L)} \right| \leq C p^{e(\dim(X) - 1)}\end{equation}
    for all $m$ and $e$. Now, by applying \cite[Theorem~4.9]{LeePandeFsignaturefunction}, there exists a constant $C'>0$ such that for all $e \geq 1$, we have $I_e (mL) = H^0 (mL)$ for $m \geq C'(p^e -1)$. Note that since $I_e \subset \Ied$, it also follows that for $m \geq C'(p^e -1)$, we have $I_e ^\Delta(mL) = H^0 (mL)$. Therefore, using \Cref{perturbationfortransofrmationrule} we see that
    \[\left| a_e ^\Delta (L) - n a_e ^\Delta(L^n) \right | \leq CC' p^{e \dim(X)}.\]
    The transformation rule now follows when we divide by $p^{e (\dim(X)+1)}$ and take a limit as $e \to \infty$.
\end{proof}

\begin{rem}
\Cref{transformationrulewithpairs} can also be deduced from the main results of \cite{CarvajalRojasFiniteTorsors} and \cite{CarvajalRojasStablerFiniteMapsSplittingPrimes}, and the case when $\Delta = 0$ and $L$ is ample was first proved in \cite[Theorem 2.6.2]{VonKorffthesis}. However, our proof gives a simpler approach to \Cref{transformationrulewithpairs} in the case of Veronese subrings and orbifold cones over projective pairs. \Cref{twistedFsig} can also be applied to studying the $F$-signature of big divisors, which will appear in an upcoming joint work with Seungsu Lee. 
\end{rem}

\subsubsection{\textbf{Restricting Weil divisors to normal, complete intersection subschemes:}} \label{restrictingWeildivisors} Let $X$ be a normal variety over $k$ and $D$ be a Weil divisor on $X$. Suppose $Y \subset X$ is a closed subvariety that is locally defined by a regular sequence in $X$. Assume that $Y$ is also normal. In this situation, we may ``restrict" the rank one reflexive sheaf $\cF:= \cO_X(D)$ on $X$ to a reflexive sheaf $\cF_Y$ on $Y$ as follows: Let $U$ be the regular locus of $Y$. Then there is an open subset $V \subset X_{\text{reg}}$ (where $X_{\text{reg}}$ denotes the regular locus of $X$) such that $V \cap Y = U$. This is possible because $Y$ is a complete intersection in $X$. Therefore, we may restrict $\cF$ to $V$ and then to an invertible sheaf on $U$, since $\cF|_V$ is invertible. Define $\cF_Y$ to be
            \[ \cF_Y :=  i_* (\cF|_U) \]
            where $i: U \to Y$ is the inclusion. Then, $\cF_Y$ is a rank one reflexive sheaf on $Y$ because $Y$ is normal and $U$ contains all the codimension one points of $Y$. Thus, we can write $\cF_Y$ as $\cO_Y (D_Y)$ for some Weil-divisor $D_Y$ on $Y$. Furthermore, if $\text{Supp}(D)$ does not contain $Y$, then since $Y$ is normal, hence integral, $D$ naturally restricts to a Cartier divisor $D_U$ on $U$ (given by restricting the equation for $D$) and we may take $D_Y$ to be the closure of $D_U$. It is also clear from the description of restriction that it commutes with addition of Weil-divisors (since the restriction of Cartier divisors on the regular locus commutes with addition).

            Now assume that $Y$ is a normal Cartier divisor on $X$. Then, restricting a Weil divisor $D$ on $X$ to $Y$ as described above, we note that we have a left exact sequence
            \[ 0 \to \cO_X(D - Y) \to \cO_X(D) \to \cO_Y (D_Y). \]
            Note that $X_{\text{reg}} \cap Y$ contains all the codimension one points of $Y$. Thus, pushing forward from $X_{\text{reg}}$, we obtain a left exact sequence
            \begin{equation} \label{restrictingtonormaldivisorexactsequence} 0 \to H^0 (X, \cO_X(D - Y) \to H^0 (X, \cO_X(D)) \to H^0 (Y, \cO_Y(D_Y)). \end{equation}

\section{The $\FA$-invariant of globally $F$-regular pairs}  \label{section3}

In analogy with Tian's $\alpha$-invariant in complex geometry, we define the ``Frobenius-$\alpha$" invariant (denoted by $\FA$) for any projective, globally $F$-regular pair $(X, \Delta)$  (\Cref{defn.GlobFreg}) and with respect to an ample Cartier divisor.

\begin{notation} \label{notation:section3}
    By a projective pair $(X, \Delta)$ over $k$, we mean that $X$ is a normal projective variety over $k$ and $\Delta$ is an effective $\QQ$-divisor on $X$. If $S$ is the section ring $S(X, L)$ of $X$ with respect to some ample line bundle $L$ on $X$, then by the cone over $(X, \Delta)$, denoted by $(S, \tilde{\Delta})$, we mean the (affine) pair $(S, \tilde{\Delta})$ where $\tilde{\Delta}$ denotes the cone over $\Delta$ with respect to $L$ (see \Cref{sectionringdfn} and \Cref{conesoverdivisorssubsection}).
\end{notation}

\subsection{Definition of the $\FA$-invariant.}

\begin{dfn} \label{Fptdfn}
    Let $(R,\fm)$ be an $F$-finite normal local ring and $\Delta$ be an effective $\QQ$-divisor on $X = \Spec(R)$. Assume that the pair $(R, \Delta)$ is strongly $F$-regular (\Cref{defn.GlobFreg}). Then, we define the $F$-\emph{pure threshold} of an effective $\QQ$-divisor $D \geq 0$ with respect to $(X, \Delta)$ to be:
   $$ \fpt(R, \Delta ; D) := \sup \{\, \lambda \, | \, (R, \Delta +  \lambda D) \text{ is sharply $F$-split} \,
   \}. $$
   See \Cref{defnsharpFsplitting} for the definition of sharp $F$-splitting.
\end{dfn}

 \begin{rem} \label{equivcharoffpt}
       When the pair $(R, \Delta)$ is strongly $F$-regular, it follows from \Cref{lemdecreasingcoefficients} that the $F$-pure threshold of $ D$ with respect to $(R, \Delta)$ is equivalently characterized as
   \[ \fpt (R, \Delta; D) = \sup \{\, \lambda \, | \, (R, \Delta+  \lambda D) \text{ is strongly $F$-regular} \} .\]
   \end{rem}

If $D$ is the principal divisor corresponding to a function $f \in R$, we write $f$ instead of $D$ in the notation for the $F$-pure threshold.

\begin{dfn} \label{FASdfn}
        Let $(X, \Delta)$ be a projective, globally $F$-regular pair and $L$ be an ample Cartier divisor on $X$. Let $(S, \tilde{\Delta})$ be the cone over $(X, \Delta)$ with respect to $L$ (see \Cref{notation:section3} 
        Then, we define
        $$ \FA(X, \Delta; L) = \inf \{ \, \fpt(S, \tilde{\Delta};  f) \, \deg(f) \, | \, 0 \neq f \in S \text{ homogeneous element} \, \}.
        $$
\end{dfn}

\begin{notation}
    We may also use $\FA(S, \tilde{\Delta})$ to denote $\FA(X, \Delta; L)$. In case $\Delta = 0$, we just write $\FA(X, L)$ and $\FA(S)$ for $\FA (X, 0; L) $ and $\FA(S, 0)$ respectively.
\end{notation}

\begin{rem}
     Suppose $S$ is an $\NN$-graded section ring. If $\tilde{\Delta}$ is a $\QQ$-divisor on $\Spec(S)$ that is the cone over a $\QQ$-divisor on $X = \Proj(S)$, then by \Cref{splitting on the cone} the $F$-pure threshold $\fpt(S, \tilde{\Delta}; f)$ for any homogeneous element $f$ of $S$ is the same as when computed after localizing at the maximal ideal $\fm$. Therefore, we will skip writing the localization at $\fm$ whenever we consider such $F$-pure thresholds.
\end{rem}

We have the following equivalent characterizations of the $\FA$-invariant:

\begin{lem} \label{equivalentdefinitions}
 Let $(X, \Delta)$ be a projective, globally $F$-regular pair and $L$ be an ample Cartier divisor on $X$. And let $(S, \tilde{\Delta})$ denote the cone over $(X, \Delta)$ with respect to $L$. Then, $\FA(X, \Delta; L)$ from \Cref{FASdfn} is equal to the supremum of any of the following sets:
 \begin{enumerate}
     \item The set of $\lambda \geq 0$ such that the pair $(S, \tilde{\Delta} + \frac{\lambda}{n} D)$ is sharply $F$-split for every $n \in \NN$ and every effective divisor $D \sim nL$.
     \item The set of $\lambda \geq 0$ such that the pair $(S, \tilde{\Delta} + \frac{\lambda}{n} D)$ is strongly $F$-regular for every $n \in \NN$ and every effective divisor $D \sim nL$.
     \item The set of $\lambda \geq 0$ such that the pair $(X, \Delta +\frac{\lambda}{n} D)$ is globally sharply $F$-split for every $n \in \NN$ and every effective divisor $D \sim nL$.
     \item The set of $\lambda \geq 0$ such that the pair $(X, \Delta+\frac{\lambda}{n} D)$ is globally $F$-regular for every $n \in \NN$ and every effective divisor $D \sim nL$.
     \item The set of $\lambda \geq 0$ such that the pair $(X, \Delta + \lambda D)$ is globally $F$-regular for every effective $\QQ$-divisor $D \sim_\QQ L$.
 \end{enumerate}
\end{lem}

\begin{proof}
    Statements (1) and (2) follow immediately from the definition of the $\FA$-invariant and the definition of the $F$-pure threshold (\Cref{Fptdfn} and \Cref{equivcharoffpt}). Statements (3) and (4) follow from (1) and (2) by using \Cref{splitting on the cone}. Part (5) is just a reformulation of (4) since every effective $\QQ$-divisor $D \sim _\QQ L$ can be written as $\frac{1}{n} nD$ so that $nD$ is an effective Cartier divisor and $nD \sim nL$.
\end{proof}

\subsection{Alternate characterizations}

Next, we prove the main result of this section which is a more precise and useful connection between the $\FA$-invariant and Frobenius splittings in $S$.

\begin{thm} \label{alternatecharofalpha}

    Let $S$ be a section ring (of some projective variety $X = \Proj(S)$) over $k$ and $\Delta$ be an effective $\QQ$-divisor that is the cone over an effective $\QQ$-divisor over $X$ with respect to $L = \cO_X(1)$. Assume that the pair $(S, \Delta)$ is strongly $F$-regular.
    \begin{enumerate}
        \item Let $\alpha^{\prime}(S, \Delta)$ denote the supremum of the set $\mathscr{A} (S, \Delta)$ defined as:
     $$  \{ \lambda \in \RR_{\geq 0} \, | \,  \text{ $\exists \, e (\lambda)> 0$ such that $\forall \, e \geq e(\lambda)$ and all $m \leq \lambda (p^e - 1)$, we have } I_e ^{\Delta}  (mL) = 0 \}. $$
    Then, $\alpha^{\prime}(S, \Delta) = \FA(S, \Delta)$.  See \Cref{Iedfnwithdelta} for the definition of $\Ied$.
    \item If there is an $e_0 > 0$ such that $(p^{e_0} - 1) \Delta$ is a $\ZZ$-Weil divisor, then $\FA(S, \Delta)$ is also the supremum of the set 
    $$ \mathscr{A}^{e_0} (S, \Delta) :=  \{ \lambda \in \RR_{\geq 0} \, | \,  \text{ $\forall \, e \geq 1$ with $e_0 |e$ and all $m \leq \lambda (p^e - 1)$, we have } I_e ^{\Delta}  (mL) = 0 \}. $$
    Moreover, $\FA(S, \Delta)$ belongs to the set $\mathscr{A}^{e_0}(S, \Delta)$, which consequently is a closed interval of the form $[0, \FA(S, \Delta)]$.
    \end{enumerate} 
\end{thm}

Recall that a $\QQ$-divisor $\Delta$ on a section ring $S$ is a cone over a $\QQ$-divisor on $X = \Proj(S)$ if and only if all the prime components of $\Delta$ are defined by homogeneous prime ideals.

\begin{rem}
    Note that in the above Proposition, it is unclear if the set $\mathscr{A}(S, \Delta)$ contains any non-zero elements. This is equivalent to the positivity of $\FA(S, \Delta)$ and will be addressed in \Cref{positivityofalpha}.
    \end{rem}

\begin{rem}
    Part (2) of \Cref{alternatecharofalpha} has a monotonicity built into the statement, and hence requires careful arguments that we present below. This is particularly useful when $e_0 = 1$ (for example, this is true when $\Delta = 0$). See \Cref{cubicsurfacecontd} for an application of this idea.
\end{rem}
 
The key idea in the proof of \Cref{alternatecharofalpha} is the following lemma which is a slight generalization of a result of Hern\'andez (also see \cite[Theorem~3.11]{SchwedeCentersOfFPurity} for a similar idea). This lemma allows us to bound the iterates of the Frobenius needed to split a given divisor, in terms of the $F$-pure threshold:
        \begin{lem} \label{fptvssplitting}
    Let $(S, \fm)$ be an $F$-finite, strongly $F$-regular local ring. Fix an effective Weil-divisor $D$ on $X = \textrm{Spec}(S)$. Then, for any fixed $e_{0} > 0$, let $\psi_{e_0}$ denote the natural map
    $$  \psi_{e
    _{0}}: \cO_X \to F_{*} ^{e_{0}} (\cO_X (D)). $$
    Then, the following are equivalent:
    \begin{enumerate}
        \item The map $\psi _{e_0}$ splits as a map of $\cO_X$-modules.
        \item The pair $(X, \frac{1}{p^{e_0} -1} D)$ is sharply $F$-split (\Cref{defnsharpFsplitting}). 
        \item The $F$-pure threshold of $ (X,D)$ is at least $ \frac{1}{p^{e_0} - 1} $ (\Cref{Fptdfn}). 
    \end{enumerate}
    In particular, $D$ splits from $F_* ^{e_0}$ if and only if $pD$ splits from $F_* ^{e_0+1}$.
\end{lem}
\begin{proof}
It follows immediately from the definitions that (1) implies (2), and (2) implies (3). Hence, it remains to show that (3) implies (1).

Following \cite[Thoerem 4.9]{HernandezFpurityofHypersurfaces}, if $\text{fpt}_{\fm}(X, D) \geq \frac{1}{p^{e_0}-1}$, we must must have that the pair $(X, \frac{1}{p^{e_0}} D)$ is sharply $F$-split (even strongly $F$-regular, by \Cref{lemdecreasingcoefficients}). Thus, there is an $e> 0$ such that the natural map
\begin{equation} \label{psidfn}
     \psi_{e} (D) : \cO_X \to \Fe(\cO_X (\lceil \frac{p^{e} - 1}{p^{e_0}} \rceil D) \end{equation} 
splits. Since the same holds for $\psi_{n e} (D)$ for any natural number $n \geq 1$ (see \cite[Proposition 3.3]{SchwedeSharpTestElements} for the proof), we get the map:
$$  \psi_{e e_0} (D) : \cO_X \to F_{*} ^{ee_0} (\cO_X ( \lceil \frac{p^{e e_0} - 1}{p^{e_0}} \rceil D) ) = F_* ^{e e_0} (\cO_X (p^{(e - 1) e_0} D))$$
splits. Note that $\psi_{e_0}$ as defined in \Cref{psidfn} matches with the map considered in the statement of the Lemma.

Let $U \subset X$ denote the regular locus of $X$. Note that since $S$ is strongly $F$-regular, it is normal, and hence $U$ contains all the codimension $1$ points of $X$. Thus, because $\phi _{e_0}$ is a map between reflexive sheaves, to show that it splits, it suffices to show that its restriction of $U$ splits. Over $U$, we may construct the map $\psi_{e e_0}$ as follows: First consider the map 
$$ \phi_{(e-1) e_0} : \cO_U (D) \to F_{*} ^{(e-1)e_{0}} \cO_U (p^{(e-1)e_0} D) $$
obtained by twisting the $(e-1)e_0 ^{\text{th}}$-iterate of the Frobenius map by the invertible sheaf $\cO_{U} (D)$. If $f$ denotes the local equation of $D$, then $\phi_{(e-1) e_0} $ is defined by sending 
$$ f \mapsto F_{*} ^{(e-1) e_0} f^{p^{(e-1)e_0}} .$$ Then, after restricting to $U$, we have that
 \[  \psi _{ e e_0} = F_* ^{e_0} \phi_{(e-1)e_0}  \circ (\psi_{e_0} |_U) \]
where the right hand side is the composition
$$ F_* ^{e_0} \phi_{(e-1)e_0}  \circ (\psi_{e_0} |_U) : \cO_U \to F_* ^{e_0} \cO_U (D) \to F_* ^{e e_0} (\cO_X (p^{(e - 1) e_0} D))  .$$
Therefore, if $\psi_{e e_0} $ splits then so does $\psi_{e_0}$. This proves that part (3) implies part (1). The above proof also shows that if $pD$ splits from $F_* ^{e_0 + 1}$, then $D$ splits from $F_* ^{e_0}$. Thus, the last  part of the lemma follows from observing that $\fpt(X, pD) = \frac{\fpt(X,D)}{p}$ and $p(p^{e_0} -1) \leq p^{e_0+1} - 1$.
\end{proof}


Next we prove a pairs version of \Cref{fptvssplitting}.

\begin{lem} \label{fptvssplittinwithdelta}
    Let $(X = \Spec(R), \Delta)$ be a local, $F$-finite, strongly $F$-regular pair, and $D$ be an effective, $\ZZ$-Weil divisor on $X$. Assume that $e_0 \geq 1$ is
    such that $(p^{e_0} -1 )\Delta$ is $\ZZ$-Weil (i.e., all coefficients of $(p^{e_0} -1) \Delta$ are integers). Then, the natural map
    \[ \psi_{e_0}: \cO_X \to F_* ^{e_0} \cO_X ( (p^{e_0} -1) \Delta + D)  \]
    splits if and only if
    $\fpt(X, \Delta; D)$ is at least $\frac{1}{p^{e_0} -1}$.

    More generally, if any of the coefficients of $\Delta$ have denominators divisible by $p$, then we still have that if the map
    \[\psi_{e}: \cO_X \to \Fe \cO_X(\lceil (p^e -1) \Delta \rceil + D)\]
    splits for some $e \geq 1$, then $\fpt(X, \Delta; D) \geq \frac{1}{p^e -1}$.
\end{lem}   

\begin{proof}
    First, suppose that $\psi_{e_0}$ splits. Then by applying Part (2) of \Cref{fptvssplitting} applied to the divisor $(p^{e_0} -1) \Delta + D$, we get that the pair $(S, \Delta+ \frac{1}{p^{e_0} -1} D)$ is sharply $F$-split. Thus, by the definition of the $F$-pure threshold (\Cref{Fptdfn}), we get that in terms of the $F$-pure threshold of pairs, $\fpt(X, \Delta; D)$ is at least $\frac{1}{p^{e_0} - 1}.$ This also proves the last part of the lemma, since we have $\lceil (p^e -1 ) \Delta \rceil + D = \lceil (p^e -1 ) (\Delta + \frac{1}{p^e -1 } D) \rceil$.
    
    Conversely, assume that $\fpt(X, \Delta; D)$ is at least $\frac{1}{p^{e_0} - 1}$. Then, we claim that
    we also have:
    \begin{equation} \label{fptlemmainequality}  \fpt(X, (p^{e_0} -1) \Delta + D) \geq \frac{1}{p^{e_0} - 1}.\end{equation}
    To prove this, it is sufficient to check that for every rational number $\varepsilon$ such that $0 < \varepsilon \ll 1$, the pair $(X, \frac{1 - \varepsilon}{p^{e_0} -1} (p^{e_0} -1 )\Delta + \frac{1 - \varepsilon}{p^{e_0} -1} D)$ is strongly $F$-regular. By \Cref{lemdecreasingcoefficients}, this follows if we show that $(X, \Delta + \frac{1- \varepsilon}{p^{e_0} - 1} D)$ is strongly $F$-regular. Since $(X, \Delta)$ is assumed to be strongly $F$-regular, this follows from \Cref{equivcharoffpt}. This proves \Cref{fptlemmainequality}. Then, applying \Cref{fptvssplitting} to the divisor $(p^{e_0} - 1) \Delta + D$, we get that the map $\psi_{e_0}$ splits. This completes the proof of the converse and hence of the lemma.
\end{proof}

We also need to following continuity properties of the $\FA$-invariant with respect to the coefficients of $\Delta$.

\begin{Pn} \label{continuityofFalpha}
    Let $(X, \Delta)$ be a projective, globally $F$-regular pair over $k$ and let $L$ be an ample line bundle over $X$. Then,
    \begin{enumerate}
        \item For every effective $\QQ$-divisor $\Delta' \leq \Delta$, we have
        \[ \FA(X, \Delta'; L) \geq \FA(X, \Delta; L). \]
        \item The $\FA$-invariant is right-continuous:
        \[ \FA(X, \Delta; L) = \sup \{ \FA(X, (1+ \varepsilon )\Delta; L) \, | \,   0 < \varepsilon \ll 1, \,  \varepsilon \in \QQ \}.\]
        \item The $\FA$-invariant is also left-continuous:
        \[ \FA(X, \Delta; L) = \inf \{ \FA(X, (1 - \varepsilon) \Delta; L) \, | \,   0 < \varepsilon \ll 1, \, \varepsilon \in \QQ \}. \]
        Together with Part (2), this shows that $\FA(X, \lambda \Delta; L)$ is a continuous function of $\lambda$.
    \end{enumerate}
\end{Pn}
\begin{proof}
    Let $S$ denote the section ring of $X$ with respect to $L$ and $\fm $ denote the homogeneous maximal ideal of $S$. Set $\alpha = \FA(X,\Delta;L)$.
    \begin{enumerate}
        \item Let $\tilde{\Delta}$ and $\tilde{\Delta'}$ denote the cones over $\Delta$ and $\Delta'$ with respect to $L$ respectively. Then, for any non-zero homogeneous element $f \in S$, since $\Delta' \leq \Delta$, we have $\fpt(S, \tilde{\Delta'};f) \geq \fpt(S, \tilde{\Delta}; f)$ by \Cref{lemdecreasingcoefficients}. Part (1) of the lemma now follows from the definition of the $\FA$-invariant.
        \item Let $\alpha^+$ denote the supremum in the statement. It follows from Part (1) that $\FA(X, \Delta;L)$ is at least $\FA(X, (1 + \varepsilon) \Delta; L)$ for each $\varepsilon > 0$. So, we have $\alpha \geq \alpha^+$.

        Now, pick (and fix) an effective divisor $D_0 \sim rL$ for some $r \gg 0$ such that $D_0 \geq \Delta$. Then, for all $0 < \varepsilon \ll 1$ (so that $\alpha - r \varepsilon > 0$), we claim that
        \begin{equation} \label{inequalityalphaepsilon}
            \FA(X,(1+ \varepsilon) \Delta; L) \geq \alpha - r\varepsilon.  \end{equation}
        To see this, let $\alpha' < \alpha$ be an arbitrary rational number and $D \geq 0$ be an arbitrary effective $\QQ$-divisor such that $D \sim_\QQ L$. Then,
        \[ (1+ \varepsilon)\Delta + (\alpha' - r\varepsilon) D \leq \Delta + \varepsilon D_0 +  (\alpha' - r\varepsilon) D \]
        and we note that $ \varepsilon D_0 + (\alpha - r \varepsilon)D$ is an effective $\QQ$-divisor linearly equivalent to $\alpha' L$. Therefore, by Part (5) of \Cref{equivalentdefinitions} (and since $\alpha' < \alpha$), we obtain that $(X, \Delta + \varepsilon  D_0 + (\alpha' - r \varepsilon) D)$ is globally $F$-regular, and thus $(X, (1+ \varepsilon)\Delta + (\alpha' - r \varepsilon )D) $ is globally $F$-regular. Since $D$ was an arbitrary effective $\QQ$-divisor with $D \sim_{\QQ} L$, by Part (5) of \Cref{equivalentdefinitions} again, we obtain that
        \[\FA(X, (1+ \varepsilon) \Delta; L) \geq \alpha' - r \varepsilon.\]
        Lastly, since $\alpha' < \alpha$ was arbitrary, this proves the inequality (\ref{inequalityalphaepsilon}).

        Finally, by taking $\varepsilon \to 0$ in \Cref{inequalityalphaepsilon}, we get $\alpha^+ \geq \alpha$, and therefore $\alpha = \alpha^+$.
        
        \item The proof of Part (3) is similar to Part (2), but with a few changes that we point out. Let $\alpha^-$ denote the infimum in the statement. Then, we similarly have $\alpha \leq \alpha^-$.

         Let $D_0 \sim r L$ be as in the proof of Part (2) and pick any rational $\varepsilon$ such that $0< \varepsilon \ll 1$. Applying \Cref{inequalityalphaepsilon} to the pairs of divisors $\frac{1}{1+ \varepsilon} \Delta$ and $\Delta$ (note that we may do this because $\frac{1}{1+ \varepsilon} \Delta \leq \Delta \leq D_0$), we get
         \[ \FA(X, \Delta; L) \geq \FA(X, \frac{1}{1 + \varepsilon} \Delta; L) - r \varepsilon.\]
         Since, $\frac{1}{1+ \varepsilon
         } \Delta < \Delta$, we have $\FA(X, \frac{1}{1+ \varepsilon} \Delta; L) \geq \alpha^-$. Taking $\varepsilon \to 0$, we obtain $\alpha = \alpha^-$. This completes the proof of the lemma. \qedhere 
    \end{enumerate}
\end{proof}

\begin{proof}[\textbf{Proof of \Cref{alternatecharofalpha}}]
    The proof of the two parts of the lemma will be intertwined. But we point out the specific parts that show the stronger claims of Part (2).

    First, we will prove the Proposition by assuming that the coefficients of $\Delta$ do not have denominators divisible by $p$. In other words, we assume that $(p^{e_0} - 1) \Delta$ is an $\ZZ$-Weil divisor on $X$ for some $e_0 >0$ (and fix $e_0$). Note that this automatically implies that $(p^e -1) \Delta$ is $\ZZ$-Weil whenever $e$ is a multiple of $e_0$. As a shorthand, set $\alpha = \FA(S, \Delta)$, and $\alpha' = \alpha'(S, \Delta)$ (which is defined in the statement of the Proposition).
    
    Now we begin by proving that $\alpha \leq \alpha'$.  This is clear if $\alpha = 0$, so we assume that $\alpha$ is positive. For any non-zero element $f \in S_m$, by definition of $\alpha$, we must have $\fpt(S,\Delta;  f) \geq \frac{\alpha}{m}$. So, for any $ e \geq 1$ such that $e_0 | e$, if $m \leq \alpha (p^e -1)$, we have
    $$ \fpt(S,\Delta; f) \geq \frac{\alpha}{m} \geq \frac{\alpha}{\alpha (p^e -1)} = \frac{1}{p^e -1}.$$ Thus, by \Cref{fptvssplittinwithdelta}, we get that $f \notin \Ied (mL)$. Since $f$ was an arbitrary non-zero element of degree $m$, this shows that $I_e (mL) = 0$ whenever $m \leq \alpha(p^e -1) $ and $e_0 | e$. Therefore, $\alpha$ belongs to the set $\mathscr{A}^{e_0}(S)$.
    
    More generally, to prove that $\alpha \leq \alpha'$, it is enough to show that $\alpha - \varepsilon \leq \alpha'$ for every $\varepsilon > 0$. Therefore, it is enough to show that for every $\varepsilon > 0$, we have $\alpha - \varepsilon \in \mathscr{A}(S, \Delta)$. Now, given an $\varepsilon > 0$,  we pick $e \gg 0$ to ensure that the following condition is true:
    there exists an effective $\QQ$-divisor $E \sim_\QQ \frac{\varepsilon}{2} L$ on $X= \Proj(S)$ such that 
    \begin{equation} \label{dominatingfractionalpart}
    (p^e -1) E \geq \text{Supp}(\Delta)_{\text{red}}, \end{equation}
    where $\text {Supp}(\Delta)_{\text{red}}$ denotes the divisor that is the sum of the prime components appearing in $\Delta$. Now, let $e' \geq e$ and pick an arbitrary non-zero element $f \in S_m$ with $m \leq (\alpha - \varepsilon) (p^{e'} -1 )$. We need to show that $f \notin I_{e'} ^\Delta (m)$. Let $D$ denote the effective divisor $\text{div}(f)$ on $X=\Proj(S)$ (and $L  \sim \cO_X (1)$). By Part (5) of \Cref{equivalentdefinitions}, we know that the pair $(X, \Delta+ \frac{1}{p^{e'} - 1} D +  D')$ is globally $F$-regular for every effective $\QQ$-divisor $D' \sim_\QQ \lambda L $, and $\lambda < \varepsilon$. This is because $D \sim mL$, so $ \frac{1}{p^{e'} - 1}D+ \lambda L \sim _\QQ (\frac{m}{p^{e'} -1 } + \lambda) L $ and by assumption, we have $\frac{m}{p^{e'}  -1} + \lambda < \alpha$. Now by \Cref{dominatingfractionalpart}, we can have the $\QQ$-divisor $E \sim_{\QQ} \frac{\varepsilon}{2} L$ such that
\[ \lceil (p^{e'} - 1) \Delta \rceil  + D \leq (p^{e'} -1 )(\Delta + \frac{1}{p^{e'} - 1} D + E).\]
Therefore, and by the previous discussion, we know that the pair $(X, \Delta + \frac{1}{p^{e'} -1} D + E)$ is globally $F$-regular. Hence, so is the pair $(X, \frac{1}{p^{e'} -1}( \lceil (p^{e'} - 1) \Delta \rceil  + D)) $ (by \Cref{lemdecreasingcoefficients}). Then, by \Cref{fptvssplitting} we get that the natural map $\cO_X \to F_* ^{e'} (\cO_X(\lceil (p^{e'} - 1) \Delta \rceil + D )$ splits, which shows that $f \notin I_{e'} ^\Delta (mL)$ as required. This completes the proof that $\alpha \leq \alpha'$.
    
    Next we prove that $\alpha \geq \alpha'$. Fix any $\lambda < \alpha'$. We will show that for all non-zero homogeneous elements $f$ of degree $m$ in $S$, we have $\fpt(S, \Delta; f) \geq \frac{\lambda}{m}$. Whenever we have $e \geq e(\lambda)$ (as defined in the statement of \Cref{alternatecharofalpha}), $m \leq \lambda (p^e -1)$ and $f \in S_m$ is any non-zero element, we know that $f$ is not contained in $I_e (mL)$ (by definition of the set $\mathscr{A}(S, \Delta)$). Furthermore, if $e_0 |e$, by \Cref{fptvssplittinwithdelta} we also have $\fpt(S, \Delta; f) \geq \frac{1}{p^e -1}$ . In other words, if $e$ is the smallest integer such that $m \leq \lambda (p^e -1)$ and $e_0 | e$, (equivalently, $e_0|e$ and $e$ can be written as $ e = \lceil \log_{p} (\frac{m}{\lambda}) \rceil + r$, for some $0 \leq r < e_0 $) and if $e \geq e(\lambda)$, then
    $\fpt(S, \Delta; f) \geq \frac{1}{p^{e} -1}.$ Combining this with the fact that $\fpt(S, \Delta; f^{a}) = \frac{\fpt(S, \Delta; f)}{a}$ for any integer $a$, we get that for all $a \gg 0$:
    \[ \fpt(S, \Delta; f) \geq \frac{a}{p^{\lceil \log_{p} (\frac{am}{\lambda}) \rceil + r_a} -1}  \]
    where $0 \leq r_a < e_0$ is such that $\lceil \log_{p} (\frac{am}{\lambda}) \rceil + r_a $ is divisible by $e_0$.
    Now, we claim that
    \begin{equation} \label{loginequality}
     \fpt(S, \Delta; f) \geq \sup_{a \gg 0} \frac{a}{p^{\lceil \log_{p} (\frac{am}{\lambda}) \rceil + r_a} -1} \geq  \frac{\lambda}{m}.\end{equation}  
 To prove the right inequality, we make the following observations: Fixing $m$ and $ \lambda$ and for any $a$, write 
 \[ \lceil \log_{p}(a) + \log_{p}(\frac{m}{\lambda}) \rceil =\log_{p}(a) + \log_{p}(\frac{m}{\lambda}) + \varepsilon(a) \]
 for some non-negative real number $\varepsilon(a)$. Then, we have
$$\inf_{a \gg 0} \frac{p^{\log_{p}(a) + \log_{p}(\frac{m}{\lambda}) + \varepsilon(a) + r_a}-1}{a} = \inf_{a  \gg 0} \,\,  \frac{m}{\lambda} p^{\varepsilon(a)+ r_a} - \frac{1}{a} \leq  \inf _{a \gg 0 } \,  \frac{m}{\lambda} p^{\varepsilon(a)+ r_a} .$$
 So it suffices to show that 
\[ \inf _{a \gg 0} p^{\varepsilon(a) + r_a} = 1.
\]
To see this, we claim that given any real number $\gamma$, the lim-inf of the sequence $\{ \lceil \gamma + \log_{p}(a) \rceil - \gamma - \log_{p}(a) \, | \, a \geq 1 \, \& \, \lceil \gamma+ \log_p (a) \rceil \text{ is divisible by $e_0$} \}$ is zero. This observation is true because the sequence $\gamma + \log_p (a)$ satisfies
\begin{enumerate}
    \item $\lim_{a \to \infty} \gamma + \log_p (a) =  \infty$, and
    \item $\lim_{a \to \infty} (\log_p (a+1) - \log_p(a)) = \lim_{a \to \infty} \log_p (1 + \frac{1}{a}) = 0 .$
\end{enumerate}

This proves the inequality (\ref{loginequality}) and consequently, that $\alpha \geq \lambda$. Since $\lambda$ was an arbitrary number smaller than $\alpha'$, we must have $\alpha \geq \alpha'
$ as well. This completes the proof that $\alpha = \alpha'$ in the case when coefficients of $\Delta$ do not have denominators divisible by $p$. We also note that to obtain $\alpha \geq \lambda$ above, the only exponents of $p$ considered are the ones divisible by $e_0$. This implies that $\alpha$ is also at least each element of the $\mathscr{A}^{e_0} (S, \Delta)$. Since we already proved that $\alpha$ belongs to $\mathscr{A}^{e_0} (S, \Delta)$, this proves Part (2) of \Cref{alternatecharofalpha}.

Finally, we consider the case when the coefficients of $\Delta$ may have denominators divisible by $p$: we choose two sequences of rational numbers $(a_n)$ and $(b_n)$ such that the following conditions hold:
\begin{enumerate}
    \item For each $n$, we have $0 < a_n \leq 1$ and $b_n >1$.
    \item $(a_n)$ is increasing and $\lim a_n =1$ and $(b_n)$ is decreasing and $\lim b_n = 1$.
    \item The coefficients of the divisors $a_n \Delta$ and $b_n \Delta$ have denominators not divisible by $p$ (for each $n$). Note that this can be achieved by picking the numerator of $a_n$ and $b_n$ to be a suitably high power of $p$. 
\end{enumerate}
Now, we observe that by the definition of the set $\mathscr{A} (S, \Delta)$ and the fact that if $\Delta \leq \Delta'$ then $I_e ^\Delta (mL) \subset I_e ^{\Delta'} (mL)$, it follows that if $\Delta \leq \Delta'$, then $\mathscr{A}(S, \Delta') \subset \mathscr{A}(S, \Delta) $ and thus $\alpha' (S, \Delta) \geq \alpha' (S, \Delta').$ Applying this, we get
\[ \alpha' (S, b_n\Delta) \leq \alpha'(S, \Delta) \leq \alpha'(S, a_n \Delta) \]
for all $n$. Now, since \Cref{alternatecharofalpha} applies to $a_n \Delta$ and $b_n \Delta$ by the previous part of the proof, we get
\begin{equation} \label{sandwichinequality} \FA(S, b_n \Delta) \leq \alpha'(S, \Delta) \leq  \FA(S, a_n \Delta) \end{equation}
for all $n$. Finally, by \Cref{continuityofFalpha}, taking the limit as $n \to \infty$, both ends in the inequality \ref{sandwichinequality} approach $\FA(S, \Delta)$. This proves that $\FA(S, \Delta) = \alpha'(S, \Delta)$ and hence concludes the proof of \Cref{alternatecharofalpha}.
\end{proof}

\section{Properties of the $\FA$-invariant} \label{section4}
In this section, we will establish some basic properties of the $\FA$-invariant based on the various characterizations proved in Section~3.
\subsection{Finite-degree approximations}

Now we prove a characterization of the $\FA$-invariant that provides us with approximations, denoted by $\alpha_e$, where the computation of each $\alpha_e$ relies on only finitely many degrees, thus giving us effective techniques to estimate the $\FA$-invariant. An analogous result for the complex $\alpha$-invariant is provided by \cite[Proposition~4.1]{BlumJonssonDeltaInvariant}. \Cref{finitedegreeapprox} establishes a limit formula for the $\FA$-invariant that is analogous to the $F$-signature (see \Cref{F-sigdfn} and the classical definition in \cite{TuckerFSigExists}).

First we define a variation of the subspace $\Ied$ that is more convenient to work with in this section.

\begin{dfn} \label{def:Iedvariant}
 	Let $X$ be a normal projective variety over $k$ and $\Delta$ be an effective $\QQ$-divisor on $X$. Let $D$ be any effective Weil-divisor on $X$. For any $e \geq 1$, we define the subset $\II_e ^{\Delta} (D) \subseteq H^0 (X, \cO_X(D))$ as 
 	\[\II_e ^{\Delta} = \left\{f \in H^0 (X, \cO_X(D)) \mid \varphi(\Fe f) = 0 \,\text{for every map } \varphi \in \Hom_{\cO_X}\big( \Fe (\cO_X(\lfloor p^e \Delta \rfloor)), \cO_X \big) \, \right \}.\]
\end{dfn}

\begin{dfn} \label{alphae}

Let $(X, \Delta)$ be a projective, globally $F$-regular pair over $k$ and $L$ be an ample divisor on $X$. For each integer $e \geq 1$, we define

\[ m_{e} (X, \Delta; L) : = \max \{ m \geq 0 \, | \, \II_e ^\Delta (mL)  = 0 \} \]
and define
\[ \alpha_{e}(X, \Delta;L) := \frac{m_e (X, \Delta; L)}{p^e} .\]
See \Cref{Iedfnwithdelta} for the definition of the subspaces $\Ied$.
\end{dfn}

\begin{thm} \label{finitedegreeapprox}
    Let $(X, \Delta)$ be a projective, globally $F$-regular pair and $L$ be an ample divisor on $X$. Then, we have
    \[  \lim _{e \to \infty} \alpha_{e} (X, \Delta; L) = \FA (X, \Delta; L). \]
    In particular, the limit exists. See \Cref{FASdfn} for the definition of $\FA$-invariant.
\end{thm}

\begin{lem} \label{lem:comparisonIes}
    Let $X$ be a normal projective variety and $\Delta \geq 0$ be an effective $\QQ$-divisor such all the coefficients of $\Delta$ are less than $1$. Suppose $D$ is any other effective Weil-divisor on $X$. Then we have
    \[ \II_e ^\Delta (D) \subseteq \Ied (D), \]
    with equality if either $(p^e -1) \Delta$ or $p^e \Delta$ is $\ZZ$-Weil.
    See \Cref{def:Iedvariant} and \Cref{Iedfnwithdelta} for the definitions of the subspaces.
\end{lem}
\begin{proof}
    Using the definitions, it is sufficient to show that if $e \geq 1$, then we always have
    \[\lfloor p^e \Delta \rfloor \leq \lceil (p^e -1) \Delta \rceil.\]
    But this is true because if $\lambda$ is a rational number strictly between $0$ and $1$, then
    $\lfloor p^e \lambda \rfloor - \lambda \leq (p^e -1 ) \lambda$,
    and since $\lambda <1$, we must have $\lceil (p^e -1 ) \lambda \rceil \geq \lfloor p^e \lambda \rfloor$. Moreover, if one of $\lfloor p^e \lambda \rfloor $ or $\lceil (p^e -1) \lambda \rceil$ is an integer, then since their difference is at most $1$, they must be equal. 
\end{proof}
  
\begin{lem} \label{monotonicity} Let $X, \Delta$ and $L$ be as in \Cref{finitedegreeapprox}. For any $e \geq 1$, we have
     \[ \alpha_{e} (X, \Delta;L) + \frac{1}{p^e} \geq \alpha_{e+1} (X, \Delta;L) + \frac{1}{p^{e+1}}. \]
\end{lem}

\begin{proof}
    Let $(S, \tilde{\Delta})$ denote the cone over $(X, \Delta)$ with respect to $L$. Note that since $S$ is the section ring of a normal variety, it is a normal domain. For any $e \geq 1$, let $0  \neq f $ be an element of $\II _e ^\Delta ((m_e + 1)L)$. Then, we claim that $f^p$ is contained in $\II_{e+1} ^{\Delta} (p(m_e + 1)L)$. The case when $\Delta = 0$ is proved in \cite[Lemma 4.4]{TuckerFSigExists}.
    
    Let $D$ be denote the divisor on $S$ defined by $f$. Then, by definition of $\II _e ^\Delta$, we know that the divisor $\lfloor p^e \Delta \rfloor + D$ does not split from $F_* ^e$. But by the last part of \Cref{fptvssplitting}, we also know that $p \lfloor p^{e}  \Delta \rfloor + pD$ does not split from $F_* ^{e+1}$. But since 
    \[ p\lfloor p^e \Delta \rfloor  \leq \lfloor p^{e+1} \Delta \rfloor, \] this proves that $0 \neq f^p \in \II_{e+1} ^{\Delta}$ and thus that 
    \[ p(m_e +1) \geq m_{e+1} + 1. \]
    Dividing both sides by $p^{e+1}$, we obtain the required inequality.
\end{proof}

\begin{proof}[\textbf{Proof of \Cref{finitedegreeapprox}}]
   The sequence $\{ \alpha_e + \frac{1}{p^e} \}_{e \geq 1}$ is decreasing as a consequence of \Cref{monotonicity}. Since it is a decreasing sequence of non-negative real numbers, the sequence converges to its infimum. Moreover, since the sequence $\frac{1}{p^{e}}$ converges to zero, the sequence $\{ \alpha_e \}_{e \geq 1}$ also converges and we have
   \[ \lim _{e \to \infty} \alpha_{e} = \inf _{e \geq 1} \{ \alpha_e + \frac{1}{p^e} \}. \]

   Let $(S, \tilde{\Delta})$ denote the cone over $(X, \Delta)$ with respect to $L$. It remains to show that the limit is equal to $\FA(S, \tilde{\Delta})$. First we prove this assuming that $(p^{e_0} - 1) \Delta$ is $\ZZ$-Weil for some $e_0 > 0$. Then, since we already know that the sequence $\alpha_{e}$ converges, it is sufficient to show that $\FA(S, \tilde{\Delta})$ is the limit of the subsequence $\alpha_{e e_0}$. Now, for any $e \geq 1$, since $(p^{ee_0} - 1) \Delta$ is also $\ZZ$-Weil, by \Cref{lem:comparisonIes} we know that $I_{ee_0} ^{\Delta} = \II_{ee_0} ^\Delta$. Recall that since $(X, \Delta)$ is globally $F$-regular, we must have that all coefficients of $\Delta$ are strictly less than $1$.
   
   Moreover, in this case we know that $\FA(S, \tilde{\Delta})$ belongs to the set $\mathscr{A}^{e_0}(S)$ from Part (2) of \Cref{alternatecharofalpha}.
   Thus, we must have $\FA(S, \tilde{\Delta})(p^{ee_0} - 1) \leq m_{ee_0} + 1 $, which can be rewritten as
   \[ \FA(S, \tilde{\Delta}) \leq (m_{ee_0} + 1) \fpt(S, \tilde{\Delta}; f) \leq  \frac{p^{ee_0}}{p^{ee_0} -1} (\alpha_{ee_0} + \frac{1}{p^{ee_0}}) = \frac{m_{ee_0} +1}{p^{ee_0} -1} . \]
   for each $e \geq 1$.
    Taking a limit over $e$, we obtain
    \[ \FA (S, \tilde{\Delta}) \leq  \lim _{e \to \infty} \alpha_e(X, \Delta;L). \]
    For the reverse inequality, setting $\alpha := \lim \alpha_e$, using that this limit is an infimum of $\alpha_e + \frac{1}{p^e}$, we note that
    \[ \alpha (p^{ee_0} -1) \leq (\alpha_{ee_0} + \frac{1}{p^{ee_0}}) (p^{ee_0} -1) < p^{ee_0} \alpha_{ee_0} + 1. \]
    By the definition of $m_{ee_0} = p^{ee_0} \alpha_{ee_0}$, the subspace $I_{ee_0} (mL)$ is equal to zero for each $m \leq \alpha (p^{ee_0} -1)  \leq m_{ee_0}$. Thus, $\alpha$ belongs to the set $\mathscr{A}^{e_0}(X, \Delta)$ defined in \Cref{alternatecharofalpha}. Since $\FA(S, \tilde{\Delta})$ is the supremum of $\mathscr{A}^{e_0} (S, \tilde{\Delta})$, we get that $\FA(S, \tilde{\Delta}) \geq \alpha$. This proves \Cref{finitedegreeapprox} when there is some $e_0 > 0$ such that $(p^{e_0} - 1) \Delta$ is $\ZZ$-Weil.
    
    The proof of the general case is a limiting argument similar to the last part of the proof of \Cref{alternatecharofalpha}, so we briefly sketch the proof and skip the details. We note again that, if $\Delta_1 \leq \Delta_2$ are two effective $\QQ$-divisors, then for all $e \geq 1$ and $m \geq 1$, $\II_{e} ^{\Delta_1} (mL)  \subset \II _e ^{\Delta_2} (mL)$. Then, choosing two sequences $(a_n)$ and $(b_n)$ as in the proof of \Cref{alternatecharofalpha} such that $a_n \Delta \leq \Delta \leq b_n \Delta$, we get that for each $n \geq 1$ and each $e \geq 1$
    \[  \alpha_e (X, a_n \Delta;L) \leq \alpha_e (X, \Delta; L) \leq \alpha_e (X, b_n \Delta;L).\]
    Then, using the fact that \Cref{finitedegreeapprox}
    applies to each $a_n \Delta$ and $b_n \Delta$, we conclude the proof by first taking a limit as $e \to \infty$ and then taking a limit as $n \to \infty$ and applying \Cref{continuityofFalpha}.
\end{proof}

\subsection{Positivity and comparison to the $F$-signature.}
Next we will show that the $\FA$-invariant is positive by comparing it to the $F$-signature (\Cref{F-sigdfn}). Recall that for any globally $F$-regular projective variety $X$ and an ample divisor $L$ over $X$, there exists a positive constant $C$ such that for any $e>0$ and any $m \geq C \, p^e$, we have
$ I_{e}(mL) = H^0 (X, mL). $ This follows from \cite[Theorem 4.9]{LeePandeFsignaturefunction}.

\begin{thm} \label{positivityofalpha}
    Let $(X, \Delta)$ be a globally $F$-regular projective pair and $L$ be an ample divisor over $X$. Then, the $\FA$-invariant $\FA(X, \Delta;L)$ is positive. Moreover, setting $\alpha = \FA(X, \Delta;L)$ and fixing a constant $C$ as discussed above (so that for any $e>0$ and any $m \geq C \, p^e$, we have
$ I_{e} ^{\Delta}(mL) = H^0 (X, mL)$), we have the following comparisons:

\begin{equation} \label{comparisonwithFsig} \frac{\vol(L) \,  \alpha ^{\dim(S)}}{\dim (S)!}  \, \leq \, \s(S, \tilde{\Delta}) \, \leq  \frac {\vol(L)}{\dim(S)!} \, \big( C^{\dim(S)} - (C - \alpha) ^{\dim(S)} \big)  \end{equation}
where $(S, \tilde{\Delta})$ denotes the cone over $(X, \Delta)$ with respect to $L$, $\s$ denotes the $F$-signature (\Cref{F-sigdfn}) and $\vol(L)$ denotes the volume of the ample divisor $L$ (as defined in \cite[Definition 2.2.31]{LazarsfeldPositivity1}).
\end{thm}

\begin{lem} \label{postivitylemma}
Fix notation as in \Cref{positivityofalpha}. Given a non-zero homogeneous element $f$ in $S$ of degree $n$, let $\lambda > n \, \fpt (S, \tilde{\Delta}; f)$ be a real number. Then,
  \begin{equation}
        \s(S, \tilde{\Delta}) \leq  \frac{\vol(L)}{\dim(S)!} \, \big( C^{\dim(S)} - (C - \lambda) ^{\dim(S)} \big) .
    \end{equation}
\end{lem}
\begin{proof}
Since we have assumed that $\lambda > n \,\fpt(S, \tilde{\Delta}; f)$, there exist integers $a,e_{0} >0$ such that 
\[\fpt(S, \tilde{\Delta}; f) < \frac{a}{p^{e_0} -1} < \frac{\lambda}{n}. \] After replacing $f$ by $f^{a}$ (so that $\fpt(S, \tilde{\Delta}; f) < \frac{1}{p^{e_0} -1}$), and by the only if part of \Cref{fptvssplittinwithdelta}), we may assume that $f \in I_{e_0} ^{\Delta}$. We also still have $n < \lambda \, (p^{e_0} -1)$. Now, since $f $ belongs to the ideal $I_{e_{0}} ^\Delta \subset S$ (\Cref{frrkkdfn}), we have $S_{m -n} \cdot f \subset I_{e_{0}} ^\Delta (mL)$ for any $m \geq n$ yielding the inequality
\begin{equation}
  \dim_{k} \frac{S_m}{I_{e_{0}}^\Delta (mL)} \leq \dim_{k} S_m - \dim_k   S_{m-n}.  
\end{equation}
for all $m \geq n$.
Furthermore, setting $v_{r} = \frac{p^{re_{0}} -1}{p^{e_0 } - 1} $ for any integer $r$, we have that $\fpt(S, \tilde{\Delta}; f^{v_r}) < \frac{1}{p^{re_0} -1}$, and so $f^{v_r} $ belongs to  $I_{re_{0}} ^{\Delta}(n v_{r} L)$. Therefore, we similarly have

\begin{equation} \label{upperboundinequality}
  \dim_{k} \frac{S_m}{I_{ne_{0}} ^\Delta(mL)} \leq \dim_{k} S_m - \dim_{k} S_{m-nv_{r}}  
\end{equation}
for all $m \geq n v_{r}$. Then, using \Cref{formulaforFsigwithdelta} we may compute the $F$-signature $\s(S, \tilde{\Delta})$ as the limit:
$$  \frac{1}{[k':k]} \, \lim _{r \to \infty } \frac{  \sum \limits _{m = 0} ^{m = Cp^{re_{0}}} \dim_{k} \frac{S_m}{I_{re_{0}} ^{\Delta}(mL)}}{p^{re_{0}\dim(S)}} \leq \frac{1}{[k':k]} \,\Bigg( \lim _{r \to \infty } \frac{  \sum \limits _{m = 0} ^{m = Cp^{re_{0}}} \dim_{k} S_m}{p^{re_{0}\dim(S)}} - \frac{  \sum \limits _{m = nv_{r}} ^{m = Cp^{re_{0}}} \dim_{k} S_{m- n v_{r}}}{p^{re_{0}\dim(S)}} \Bigg),$$
where $k'$ is the field $S_0$.
Here we have used \Cref{upperboundinequality} and the defining property of the constant $C$. Finally, calculating the dimensions in the above inequality as $r \to \infty$ by using the asymptotic Riemann-Roch formula
$$ \dim_{k} S_m = [k':k] \, \frac{\vol(L)}{(\dim S-1) !} m^{\dim S -1} + o(m^{\dim S -2}), $$
we obtain
$$   \s(S, \tilde{\Delta}) \leq \frac{\vol(L)}{\dim(S)!} \big( C^{\dim(S)} - (C - \frac{n}{p^{e_{0}} -1}) ^{\dim(S)} \big).  $$
The proof of the lemma is now complete by using the fact that $\lambda \geq \frac{n}{p^{e_0} -1}$.
\end{proof}

\begin{proof}[\textbf{Proof of \Cref{positivityofalpha}}]
We note that if $\FA(X, \Delta;L) = 0$, then the rightmost inequality of \Cref{comparisonwithFsig} implies that the $F$-signature $\s(S, \tilde{\Delta}) $ is zero. But this is a contradiction since the pair $(S, \tilde{\Delta})$ was assumed to be strongly $F$-regular (see \cite[Theorem 0.2]{AberbachLeuschke} and more generally, \cite[Theorem~3.18]{BlickleSchwedeTuckerFSigPairs1}). So the positivity of $\FA(X, \Delta;L)$ follows from \Cref{comparisonwithFsig}, which we will now prove.

The rightmost inequality follows from \Cref{postivitylemma} by taking a limit as $\lambda \to \FA(X, \Delta; L)$, since the Lemma applies to each $\lambda $ such that $\lambda > \fpt(S, \tilde{\Delta}; f) \, \deg(f)$ for some non-zero $f$.

To prove the leftmost inequality, we first claim that a variant of \Cref{formulaforFsigwithdelta} can be used to compute the $F$-signature of $(S, \tilde{\Delta})$:
\[  \s(S, \tilde{\Delta}) = \frac{1}{[k':k]} \lim _{e \to \infty } \frac{  \sum \limits _{m = 0} ^{m = Cp^{e}} \dim_{k} \frac{S_m}{\II_{e} ^\Delta (mL)}}{p^{e\dim(S)}}.  \]
To show this formula holds, we apply \cite[Lemma~4.7]{BlickleSchwedeTuckerFSigPairs1} and check that we can pick a homogeneous element $c \in S$ such that $\text{div}(c) \geq \lceil \Delta \rceil $ so that for each $e \geq 1$, we have
\[ \lfloor p^e \Delta \rfloor \leq \lceil (p^e -1 ) \Delta \rceil  \leq \lfloor p^e \Delta \rfloor + \text{div}(c), \]
where the first inequality follows from \Cref{lem:comparisonIes}. Using the above formula for the $F$-signature of $(S, \tilde{\Delta})$, we get
\begin{equation} \label{leftsideinequality}
\s(S, \tilde{\Delta}) = \frac{1}{[k':k]} \lim _{e \to \infty } \frac{  \sum \limits _{m = 0} ^{m = Cp^{e}} \dim_{k} \frac{S_m}{\II_{e} ^{\Delta}(mL)}}{p^{e\dim(S)}} \geq \frac{1}{[k':k]}   \lim _{e \to \infty } \frac{  \sum \limits _{m = 0} ^{m = m_e} \dim_{k} S_m} {p^{e\dim(S)}}. \end{equation}
Recall that for any $ e \geq 1$, $m_e (S, \tilde{\Delta}) $ is the largest integer $m$ such that $I_e (mL) = 0$, which justifies \Cref{leftsideinequality}. But the right hand side is equal to 
\[  \lim_{e \to \infty} \frac{\vol(L)}{\dim(S) !} \big(\frac{m_e}{p^e} \big)^{\dim(S)}. \]
We conclude the proof by using \Cref{finitedegreeapprox}, which says that $\lim _{e \to \infty} \frac{m_e}{p^e} = \FA(X, \Delta;L)$. This concludes the proof of \Cref{comparisonwithFsig} and thus of \Cref{positivityofalpha}.
\end{proof}

\begin{rem}
    The positivity of the $\FA$-invariant proved in \Cref{positivityofalpha} can also be deduced from the main theorem of \cite{SatoStabilityofTestIdeals}.
\end{rem}

\begin{rem}
    It is interesting to note that constants $C$ and $\vol(L)$  involved in \Cref{comparisonwithFsig} have certain uniformity properties, for example in flat families. This fact will be applied in forthcoming works to relate properties of the $F$-signature and the $\FA$-invariant. 
\end{rem}

\subsection{Behaviour under certain ring extensions.}

In this subsection, we record some useful results on the behaviour of the $\FA$-invariant under suitably nice extensions of section rings.

\begin{Pn} \label{Prop:Veronese}
    Let $(X, \Delta)$ be a projective, globally $F$-regular pair, and $L$ be an ample divisor over $X$. Then, for any $n \geq 1$, we have
    \[ \FA(X, \Delta; L^n) = \frac{1}{n} \FA(X, \Delta;L). \]
\end{Pn}

\begin{proof}
    First we note that by \Cref{positivityofalpha}, we may assume that in the notation of \Cref{alphae}, that $m_e (X, \Delta;L) \to \infty$ as $e \to \infty$ (and similarly for $L^n$). Now pick an integer $r \gg 0$ and an effective divisor $E \sim rL$. Then, by applying \Cref{lem:addingeffective}, we see that for all $e \gg 0$ we have \[ n \,m_e (X, \Delta; nL) - r \leq m_e (X, \Delta; L) \leq n \, m_e (X, \Delta;nL) +n. \]
    Now, the Proposition follows immediately from \Cref{finitedegreeapprox}.
\end{proof}

\begin{rem}
    The $F$-signature of section rings also transforms in a similar manner to the $\FA$-invariant above. Indeed, see \cite[Theorem~2.6.2]{VonKorffthesis} and its generalization in \cite[Theorem~4.8]{CarvajalRojasFiniteTorsors}. A simple proof of this transformation rule for section rings can also be obtained by using \Cref{twistedFsig}.
\end{rem}

\begin{Pn} \label{directsummand}
    Let $S$ and $S^{\prime}$ be two $\NN$-graded section rings (of possibly different varieties) and $S \subset S^{\prime}$ be an inclusion such that for a fixed integer $n \geq 1$ and any other $m$, all degree $m$ elements of $S$ are mapped to degree $n  m$ elements of $S^{\prime}$. Further, assume that the inclusion $S \subset S^{\prime}$ splits as a map of $S$-modules.
    Then, we have
    \begin{equation} \label{directsummandequation} \FA(S) \geq \frac{\FA(S^{\prime})}{n}. \end{equation} 
 \end{Pn}

\begin{proof}
    This follows immediately from \Cref{finitedegreeapprox} and the fact that a homogeneous element $f$ of $S$ splits from $\Fe S$ if it splits from $\Fe S^{\prime}$. In other words, we have
    \[ I_e (S_m) \subset I_e (S' _{mn})\]
    for any $e$ and $m$. Here, $I_e (S_m)$ is the same as $I_e (mL)$ from \Cref{Iedfnwithdelta}, where $X = \Proj (S)$, $L = \cO_X (1)$ and $\Delta= 0$ (and similarly for $S'$).
\end{proof}

\begin{Pn} \label{regularfibers}
    Let $(S, \fm) \subset (S', \mathfrak{n})$ be a degree preserving map of $\NN$-graded, strongly $F$-regular section rings (of possibly different varieties and over possibly different perfect fields). Assume that both $S$ and $S'$ are generated in degree one, $S'$ is flat over $S$, and that the ring $S'/\fm S'$ is regular. Furthermore, suppose $\Delta$ is an effective $\QQ$-divisor over $\Spec(S)$ that is defined by homogeneous prime ideals. Then, we have
    $ \FA(S') = \FA (S)$,
    and similarly, \[ \FA(S, \Delta) = \FA(S', \varphi^* \Delta )\]
    where $\varphi: \Spec(S') \to \Spec(S)$ is the map induced by the inclusion $S \subset S'$.
\end{Pn}

Recall that since $\varphi$ is flat, we may pull-back Weil-divisors to Weil-divisors.

\begin{proof}
    First note that since $S$ and $S'$ are generated in degree $1$, the $\FA$-invariant must be at most $1$ for both of them (by \Cref{lem:coefficientmorethanone}). Next, by \Cref{continuityofFalpha}, it is enough to prove the theorem when the coefficients of $\Delta $ have denominators coprime to $p$. Then, fix $e_0$ such that $(p^{e_0} - 1) \Delta$ is $\ZZ$-Weil. Then,
 for any $e \geq 1$ such that $e_0$ divides $e$, and $m \leq p^e -1$, it follows from the \cite[Proof of Lemma~3.5]{CRSTBertiniTheorems} that
    \[ I_e ^\Delta (S_ m) \, S' = I_e ^{\varphi^* \Delta} (S'_m).   \]
    Again, $I_e ^\Delta (S_m)$ is the same as $I_e ^{\tilde{\Delta}} (mL)$ from \Cref{Iedfnwithdelta}, where $X = \Proj (S)$, $L = \cO_X (1)$ and $\tilde{\Delta}$ is the cone over $\Delta$ (and similarly for $S'$). This implies that $ m_e (S, \Delta) = m_e (S', \varphi^* \Delta) $ for any $e \geq 1$ divisible by $e_0$ (see \Cref{alphae}).
    Now, the Proposition follows by using \Cref{finitedegreeapprox}.
\end{proof}

Recall that $k$ was assumed to be a \emph{perfect} field of characteristic $p >0$.

\begin{Cor} \label{changeofbasefield}
    Let $(X, \Delta)$ be a projective, globally $F$-regular pair over $k$ and $K$ be an arbitrary perfect field extension of $k$. Then, the base-change $X \otimes _k K $ is isomorphic to a disjoint union of globally $F$-regular pairs $(X_i, \Delta_i)$ over finite extensions of $K$. Moreover, for any ample divisor $L$ and for each $i$, we have
    \[ \FA(X, \Delta;L) = \FA(X_i, \Delta_i; L_i) \]
    where $L_i$ denotes the pull-back of $L$ to $X_i$.
\end{Cor}

\begin{proof}
    Firstly, we consider the section ring $S = S(X,L)$ and we may assume that $S$ is generated in degree one by using \Cref{Prop:Veronese}. Set $S' = S \otimes _k K$. Since $k$ is perfect and $S/\fm$ is a finite separable extension of $k$, we see that
    \[ S'/\fm S' \isom S/\fm \otimes _k K \isom \prod _{i} L_i  \]
    is a finite product of perfect fields $L_i$. Thus, $S' \isom \prod _i S^i$ where $S^i$ is the section ring of $X_i := X \times _{S/\fm} L_i$. Note that $S^i$ is isomorphic to $ S \otimes_{S/\fm} L_i$, and hence each $S^i$ is strongly $F$-regular by \cite[Theorem~3.6]{CRSTBertiniTheorems}.
    Similarly, if $\Delta_i$ denotes the pull-back to $\Delta$ to $X_i$, then
    the corollary now follows from \Cref{regularfibers} by applying it to each inclusion $ S \subset S^i$.
\end{proof}

\begin{rem} \label{reductiontogeometricallyconnected}
    In the setting of \Cref{changeofbasefield}, the same results also holds for the $F$-signature of $S$ instead of the $\FA$-invariant by using \cite[Theorem~5.6]{YaoObservationsAboutTheFSignature} or \cite[Theorem~3.6]{CRSTBertiniTheorems} in place of \Cref{regularfibers}. This allows to reduce computing both the $F$-signature and the $\FA$-invariant of a section ring to the geometrically connected case (i.e., $S_0 = k$). Moreover, we may also assume that $k$ is algebraically closed by passing to the algebraic closure.
\end{rem}

\subsection{Behaviour under changing the line bundle.}

Now, we prove a monotonicity result for the $\FA$-invariant under varying the ample line bundle $L$  ``in the effective direction".

\begin{Pn} \label{Prop:Addingeffective}
    Let $(X, \Delta)$ be a projective, globally $F$-regular pair over $k$. Suppose $L$ and $L'$ are ample divisors over $X$ such that their difference $L' - L$ is effective. In other words, there is an effective Cartier divisor $0 \leq E \sim L' - L$. Then, $\FA(X, \Delta;L') \leq \FA(X, \Delta; L)$.
\end{Pn}
\begin{proof}
    By \Cref{lem:addingeffective}, for each $e \geq 1$ and each $m \geq 1$, we have $\mathbb{I}_e ^\Delta (mL) \subset \mathbb{I} _e ^\Delta (mL')$. Note that even though \Cref{lem:addingeffective} is written for the subspaces $\Ied$, the exact same proof works for the variants $\mathbb{I}_e ^\Delta$ as well. Therefore, we obtain that for all $e \geq 1$, $m_e (X, \Delta;L) \geq m_e(X,\Delta;L')$ (see \Cref{alphae}). The Proposition now follows from \Cref{finitedegreeapprox}.
\end{proof}

\begin{rem}
    While \Cref{Prop:Addingeffective} and \Cref{directsummand} are easy consequences of the theory of the $\FA$-invariant developed so far, there are no corresponding results for the $F$-signature, and the direct analogies are seen to be false. However, combined with \Cref{positivityofalpha}, we obtain interesting information about the $F$-signature as well, using the statements for the $\FA$-invariant. This observation can be used to prove certain uniform lower bounds on the $F$-signature, and will be explored in a forthcoming project.
\end{rem}

\section{The $\FA$-invariant of Fano varieties and log Fano pairs.}

In this section, we specialize the study of the $\FA$-invariant to the case 
of globally $F$-regular Fano varieties (and when the ample divisor is a multiple of $-K_X$), and more generally to log Fano pairs. We begin by defining what we mean by a \emph{globally $F$-regular log Fano pair} in positive characteristic. Recall that $k$ denotes a perfect field of characteristic $p >0$.

\begin{dfn} \label{logFanodfn} \label{QFanodfn}

     A \emph{globally $F$-regular log Fano pair} $(X, \Delta)$ consists of a projective variety $X$ over $k$, and a $\QQ$-Weil divisor $\Delta \geq 0$ such that
    \begin{enumerate}
        \item $(X, \Delta)$ is globally $F$-regular (\Cref{defn.GlobFreg}).
        \item $K_{X} + \Delta$ is a $\QQ$-Cartier divisor.
        \item $-(K_{X}+ \Delta)$ is ample.
    \end{enumerate}
    If $(X,0)$ is a globally $F$-regular log Fano pair, then we call $X$ a globally $F$-regular \emph{$\QQ$-Fano} variety.
    
\end{dfn}

Note that since $X$ has only strongly $F$-regular singularities, $X$ is automatically normal and Cohen-Macaulay. In particular, we may define the canonical Weil-divisor $K_{X}$ by extending a canonical divisor from the smooth locus. Moreover, since $X$ is Cohen-Macaulay, $\om _{X} = \cO_{X}(K_{X})$ is a dualizing sheaf over $X$.
Furthermore, the second and third conditions in Definition~\ref{logFanodfn} guarantee that there is an integer $r$ such that $r(K_{X} + \Delta)$ is Cartier, and so that $-r(K_{X} + \Delta)$ defines an ample line bundle over $X$. The smallest such $r$ is called the \emph{index} of $(X, \Delta)$.



\begin{dfn} \label{FAXdfn}
    Let $(X, \Delta)$ be a globally $F$-regular log Fano pair over $k$ and $r$ be a positive integer divisible by the index of $(X, \Delta)$. Let $L = -r(K_X + \Delta)$ denote the corresponding ample divisor. Then, the $\FA$-\emph{invariant} of $(X, \Delta)$ is defined to be
    $$ \FA (X, \Delta) := r \, \FA (X, \Delta; L ) $$
    where $\FA$ denotes the $\FA$-invariant as defined in Definition~\ref{FASdfn}.
\end{dfn}

Similarly, we also define the global $F$-signature of a log Fano pair.

\begin{dfn} \label{FsigXdfn}
    Let $(X, \Delta)$ be a globally $F$-regular log Fano pair over $k$ and $r$ denote a positive integer divisible by the index of $(X, \Delta)$. Let $L$ denote the ample divisor $-r(K_X + \Delta)$ and 
    $$ S := S(X, L) = \bigoplus _{m \geq 0} H^{0}(X, \cO_{X}(mL)) $$ denote the section ring of $X$ with respect to $L$, and let $\tilde{\Delta}$ denote the cone over $\Delta$ with respect to $L$ (see \Cref{conesoverdivisorssubsection}). Then, the \emph{$F$-signature} of $(X, \Delta)$ is defined to be
    $$ \s(X, \Delta) := r \, \s (S, \tilde{\Delta}) $$
    where $\s$ denotes the $F$-signature of pairs, as defined in \Cref{F-sigdfn}.
\end{dfn}

\begin{rem}
    Though the definitions of the $\FA$-invariant and the $F$-signature involve making a choice of a multiple of the index of $(X, \Delta)$, both invariants are well-defined thanks to \Cref{Prop:Veronese} (for the $\FA$-invariant) and \Cref{transformationrulewithpairs} (for the $F$-signature). Moreover, when $-K_X - \Delta$ is a linearly equivalent to a $\ZZ$-Weil divisor $L_0$ (for example when $\Delta = 0$), by \Cref{transformationrulewithpairs} the $F$-signature of $(X, \Delta)$ can also be interpreted as the $F$-signature of the orbifold cone over $(X, \Delta)$ with respect to $L_0$ (\Cref{def:Fsigoforbifoldcones}). In particular, this implies that in that case the $F$-signature of $(X, \Delta)$ is at most $1$.
\end{rem}

\subsection{Main theorems and proofs.}

\begin{thm} \label{alphathm}
    Let $(X, \Delta)$ be a globally $F$-regular log Fano pair of positive dimension. Then, 
 $\FA (X, \Delta)$ is at most $1/2$.
\end{thm}

The proof of \Cref{alphathm} relies on the duality for the Frobenius morphism (see \Cref{section:duality}) and its consequences, namely, \Cref{dualIe}. Since the proof of the general case involves some technical perturbation tricks, we first present a sketch of the proof in the case where $\Delta = 0$ and $-K_X$ is an ample Cartier divisor, which we hope will make the key idea transparent.
\medskip

\paragraph{\textbf{Sketch proof of \Cref{alphathm}:}}
For simplicity, assume that $-K_X$ is Cartier and ample, $p$ is an odd prime, $\Delta = 0$, and $k$ is algebraically closed. For any $m \geq 1$ and $e \geq 1$, the number $a_e (m) = \dim_k \frac{H^0 (X, \cO_X(-mK_X))}{I_e (-mK_X)}$ is equal to the maximum number of copies of $\cO_X$ that appear as summands of the sheaf $\Fe \omega_X ^{-m}$ (see \cite[Proof of Theorem~4.7]{LeePandeFsignaturefunction} for a justification). Then, for any $e \geq 1$, since we know that $\Fe \omega_X^{-m}$ and $\Fe \omega_X ^{m+1- p^e}$ are dual (as $\cO_X$-modules) by \Cref{section:duality}, they must have the same number of copies of $\cO_X$ appearing as summands. In particular, we must have
\[\dim_k \frac{H^0 (X, \omega_X^{-(\frac{p^e -1}{2} -1) })}{I_e (-(\frac{p^e-1}{2} - 1)K_X)} = \dim_k \frac{H^0 (X, \omega_X ^{-(\frac{p^e-1}{2} +1)})} {I_e(-(\frac{p^e-1}{2} +1)K_X)}.  \]
Now, since $\omega_X ^{-1}$ is ample, for $e \gg 0$, we must have $\dim_k H^0 (\omega_X ^{-(\frac{p^e -1}{2} -1)}) < \dim_k (H^0 (\omega_X ^{-(\frac{p^e -1}{2} +1)}))$, which implies that $I_e (-(\frac{p^e -1}{2} + 1)K_X) \neq 0$ for $e \gg 0$. This implies that $\FA(X) \leq 1/2$ by \Cref{alternatecharofalpha}. \qedhere

\begin{proof}[\textbf{Proof of \Cref{alphathm}}]
Firstly, by using \Cref{changeofbasefield}, we may assume that $k$ is algebraically closed and that $X$ is geometrically connected over $k$. In other words, we have $H^0 (X, \cO_X) = k$.

Let $d$ denote the dimension of $X$, and $r \gg 0 $ be an integer divisible by the index of $(X, \Delta)$, $r \Delta$ is $\ZZ$-Weil and 
such that $H^0 (X, \cO_X(-mr(K_X + \Delta)) \neq 0$ for all $m \geq 1$. Let $\cL = \cO_X( -r(K_X + \Delta))$ denote the corresponding ample line bundle and $L$ denote any Cartier divisor linearly equivalent to $\cL$.


Next, having chosen an $r$ as above, we pick an integer $n_1$ such that $H^0 (X, \cO_X(sK_X) \otimes \cL^{n_1}) \neq 0$ for all integers $s$ such that $-r \leq s \leq r$. Again, this is possible since $\cL$ is ample and since we are considering only finitely many sheaves. 

Finally, we pick an integer $n_2 >0$ such that
\[ \dim_k H^0 (X, \cL^m) < \dim_k H^0 (X, \cL^{m+n}) \]
for all $n \geq n_2 $ and all $m \geq 0$.
This is clear, since by the asymptotic Riemann-Roch formula, $\dim_k H^0 (X, \cL^m)$ is eventually a polynomial in $m$ of degree $d>0$ and a positive leading term (because $\cL$ is ample).


Assume, for the sake of contradiction, that $\FA(X) > \frac{1}{2} + \varepsilon$ for some small $\varepsilon > 0$. Then, note that by \Cref{alternatecharofalpha}, for all $e \gg 0$, we have
\begin{equation} \label{vanishingIeequation}
  \Ied(-mr (K_X + \Delta)) = 0 \, \, \text{for all } m \leq \frac{p^{e} - 1}{2r} + \frac{\varepsilon}{r} (p^e -1).  
\end{equation}

Now, for all $e \gg 0$, we can find an integer $m$ satisfying the following properties:
\begin{enumerate}
    \item $mr < \frac{p^e - 1}{2}$.
    \item $ p^e -1 - mr + n_1 r < \frac{p^e - 1}{2} +  \varepsilon (p^e -1) $, 
    \item $n_2 r + r \leq p^e -1 - 2mr - n_1r$.
\end{enumerate}
This is equivalent to finding an integer $m$ such that 
\[ \frac{p^e -1}{2r}  - \frac{\varepsilon}{r} (p^e -1) + n_1 r \leq m \leq \frac{p^e -1}{2r} - \frac{n_1}{2r} - \frac{n_2}{2r} - 1,
\]
which is possible since $r, n_1,$ and $n_2$ are fixed and $\frac{\varepsilon}{2} (p^e -1) \to \infty$ as $e \to \infty$.

\paragraph{\textbf{Claim:}} We claim that the above conditions on $e$ and $m$ guarantee that, setting $D = (1 - p^e) K_X  - \lceil (p^e - 1) \Delta \rceil - mL$, we have:
\begin{enumerate}
    \item $ \Ied(mL) = 0$,
    \item   $\Ied(D) = 0$, and
    \item $\dim_k H^0 (X, \cO_X (D)) > \dim_k H^0 (X, \cO_X(mL)) $. 
    \end{enumerate}
 Assuming this claim,  we apply \Cref{dualIe} to the divisor $ mL$ to note that
\[ \dim_{k} \frac{H^0 (X, \cO_X(mL))}{\Ied (mL)} = \dim_k  \frac{H^0 (X, \cO_X(D))}{\Ied(D)},  \]
which is in contradiction to the claim. This proves that $\FA(X, \Delta)$ is at most 1/2.

 Thus, it remains to check the three conditions from the claim above.
 \begin{enumerate}
     \item By \Cref{vanishingIeequation} we see that $\Ied(mL) = 0$ if $m \leq \frac{p^e - 1}{2}$.
     \item First we check that since $r\Delta$ is 
     assumed to be a $\ZZ$-Weil divisor, we have $\lceil (p^e - 1) \Delta \rceil  = \lfloor \frac{p^e - 1}{r} \rfloor r \Delta + E$ for some effective Weil-divisor $E$. Furthermore, we note that since we have 
     \[ p^e - 1  - \left( \lfloor \frac{p^e -1}{r} \rfloor r \right)  \leq r  \]
     and by our choice of $n_1$, we may pick an effective divisor $E'$ such that
      \[ (1-p^e)K_X  + E' \sim \left( - \lfloor \frac{p^e -1}{r} \rfloor r K_X \right ) + n_1 L. \]
     Combining these two observations, we see that
      \[  (1-p^e)K_X -  \lceil (p^e - 1) \Delta \rceil  + E+E' \sim -\left( \lfloor \frac{p^e -1}{r} \rfloor r K_X + \lfloor \frac{p^e - 1}{r} \rfloor r \Delta \right )  + n_1 L \sim m'L,\]
      where $m' = \lfloor  \frac{p^e -1}{r} \rfloor + n_1$.
     By the second condition on $m$ and \Cref{vanishingIeequation}, we see that $\Ied((m' - m)L) = 0$. Since $E+E'$ is an effective Weil divisor, by applying \Cref{lem:addingeffective} we get that  $\Ied((1 - p^e) K_X  - \lceil (p^e - 1) \Delta \rceil - mL) = 0$.
     \item Similarly, in the other direction, we have $\lceil (p^e - 1) \Delta \rceil + E =  \lceil \frac{p^e - 1}{r} \rceil r\Delta$ for some effective Weil-divisor $E$. And since $\lceil  \frac{p^e -1}{r} \rceil r - ( p^e -1) \leq r $, we may choose an effective divisor $E'$ such that
     \[  (1-p^e)K_X -  \lceil (p^e - 1) \Delta \rceil - E - E' - mL  \sim -\left( \lceil \frac{p^e -1}{r} \rceil r K_X + \lceil \frac{p^e - 1}{r} \rceil r \Delta \right )  - mL -  n_1 L \sim m'L,\]
     where $m'$ now denotes $\lceil \frac{p^e -1 }{r} \rceil - m - n_1 $. By the third condition on $m$, we know that $m'- m \geq n_2$ and hence by the choice of $n_2$ and the fact that $E$ and $E'$ are effective Weil-divisors, we have
     \[ \dim_k H^0 (X, \cO_X(D)) \geq \dim_k H^0 (X, \cO_X(m'L)) > \dim_k H^0 (X, \cO_X(mL)).\]
 \end{enumerate}
 This completes the proof of the claim and hence of \Cref{alphathm}.
\end{proof}

For the rest of this section, we assume that $H^0 (X, \cO_X) = k$, i.e., that $X$ is geometrically connected. However, see \Cref{reductiontogeometricallyconnected} for ways to extend the results to more general cases. In the case of Fano varieties and log Fano pairs, we have a stronger version of comparison of the $\FA$-invariant to the $F$-signature than the formula in \Cref{positivityofalpha}:

\begin{thm} \label{alphavss}
    Let $(X, \Delta)$ be a $d$-dimensional globally $F$-regular log Fano pair over $k$. Assume that $H^0 (X, \cO_X) = k$ and that $d$ is positive. Set $\alpha = \FA(X, \Delta)$ (\Cref{FAXdfn}). Then, we have the following inequalities relating the $F$-signature and the $\FA$-invariant:
    
    \begin{equation} \label{alphavssequation}
        \frac{2 \, \alpha ^{d+1} \,  \vol(X,\Delta)}{(d+1)!} \leq \s(X, \Delta) \leq  \frac{2 \, \big( (\frac{1}{2})^{d+1} - (\frac{1}{2} - \alpha) ^{d+1} \big) \, \vol(X, \Delta)}{(d+1)!}.
    \end{equation}
   
\end{thm}

\begin{notation}
     The \emph{volume} $\vol(X, \Delta)$ of a log Fano pair $(X, \Delta)$ is the volume of the $\QQ$-Cartier
    divisor $-K_X - \Delta$. Note that since $-K_X - \Delta$ is assumed to be $\QQ$-ample, the volume is also equal to the top intersection number $(-K_X - \Delta)^{d}$, where $d = \dim(X)$.
\end{notation}

\begin{Cor} \label{alphavsscor}
     Let $(X, \Delta)$ be a globally $F$-regular log Fano pair of dimension $d>0$. Then,
      \begin{enumerate}
        \item We have
       $$ \s(X, \Delta) \leq \frac{\vol(X, \Delta)}{2^{d} (d+1)!}.$$
       
        \item Moreover, $\FA (X, \Delta)$ is equal to $ 1/2$ if and only if the value of the $F$-signature 
        $ \s(X, \Delta)$ is equal to 
        $$ \frac{\vol(X, \Delta)}{2^{d} (d+1)!}.$$
       
    \end{enumerate}
\end{Cor}

These results rely on a more refined formula for the $F$-signature than \Cref{formulaforFsigwithdelta} that holds for $\QQ$-Fano varieties and log Fano pairs, which we prove first:

\begin{Pn} \label{Formulafors}
    Let $(X, \Delta)$ be a globally $F$-regular log Fano pair over $k$. Assume that $H^0 (X, \cO_X) = k$. Let $r$ be a positive integer such that $r\Delta$ is $\ZZ$-Weil and $r(K_X+ \Delta)$ is Cartier. 
Then, the $F$-signature of $(X, \Delta)$ can be computed as
$$ \s(X, \Delta) = \lim _{e \to \infty } \frac{ 2r \,  \sum \limits _{m = 0} ^{\lfloor \frac{p^e - 1}{2r}\rfloor}  \dim_{k} \frac{H^{0}(-mr(K_{X}+ \Delta))}{\Ied(-mr(K_{X}+ \Delta))}}{ p^{e(\dim(X)+1)}},$$
where $\Ied$ denotes the subspace defined in \Cref{Iedfnwithdelta}.
\end{Pn}

\begin{lem} \label{stoppingfreerank} Let $(X, \Delta)$ be a globally $F$-regular log Fano pair that is geometrically connected over $k$. Then, for any $r>0$ such that $r \Delta$ is $\ZZ$-Weil and $-r(K_X + \Delta)$ is Cartier, there exists an integer $n \geq 0$ (depending on $r$) such that for all $e \geq 1$, we have $H^{0}(-mr(K_{X}+ \Delta)) = \Ied(-mr(K_{X}+ \Delta)) $ whenever $m > \frac{p^e - 1}{r}+n$. \end{lem}
\begin{proof} 
The argument is similar to the proof of \Cref{alphathm}. Let $L$ denote the Cartier divisor $-rK_X - r\Delta$, and $n$ be such that $H^0 (X, \cO_X(sK_X + nL)) \neq 0$ for all $0 \leq s \leq r$. Let $m > \frac{p^e -1}{r} + n$. 

By definition of the subspace $\Ied$ (\Cref{Iedfnwithdelta}), it is sufficient to show that there are no non-zero maps $\phi: \Fe( \cO_X(\lceil (p^e -1) \Delta \rceil + mL)) \to \cO_X$. By duality for the Frobenius map (\Cref{dualityiso}), we have 
\[ \Hom_{\cO_X}( \Fe (\cO_X(\lceil (p^ e -1 ) \Delta \rceil + mL)) , \cO_X) \isom H^0 (X, (1-p^e)K_X - mL - \lceil (p^e -1 ) \Delta \rceil ).\]
But we see that $\lceil (p^e -1) \Delta \rceil \geq \lfloor \frac{p^e -1}{r} \rfloor r \Delta $. And by the choice of $n$, we can choose an effective divisor $E$ such that $(1- p^e)K_X = -r\lfloor  \frac{p^e -1 }{r} \rfloor K_X - E + nL$. Therefore, we have
\[ (1-p^e)K_X - mL - \lceil (p^e -1 ) \Delta \rceil \leq (\lfloor \frac{p^e -1 }{r} \rfloor -m + n)L . \]

Finally, by the assumption on $m$, we have $\lfloor \frac{p^e -1}{r} \rfloor + n - m < 0$, and since $L$ is ample, this implies that $H^0 (X, \cO_X((\lfloor \frac{p^e -1 }{r} \rfloor + n - m)L)) = 0 $ as required. Hence, there are no non-zero maps $\phi: \Fe \cO_{X}(\lceil (p^e -1 )\Delta \rceil +mL) \to \cO_X$, which proves the lemma.
\end{proof}

\begin{proof}[\textbf{Proof of \Cref{Formulafors}}]
Let $(S, \tilde{\Delta})$ denote the cone over $(X, \Delta)$ with respect to $L := -r(K_X+ \Delta)$. Fix an $e >0$ and let $a_{e} ^\Delta$ denote the $\Delta$-free rank of $\Fe S$ as an $S$-module (\Cref{frrkkdfn}). Recall that by \Cref{formulaforFsigwithdelta}, we have
\begin{equation} \label{Fanofreerankformula}
    a_{e} ^\Delta = \sum _{m = 0} ^{\infty} \dim_{k} \frac{H^{0}(mL)}{\Ied(mL)}
\end{equation}
so that
\[ \s(X, \Delta) = r \, \lim _{e \to \infty} \frac{a_e ^\Delta}{p^{e (\dim(X) + 1)}}  .\]
Then, by \Cref{stoppingfreerank} there is an $n>0$ such that for any $e \geq 1$, the terms of the sum in \Cref{Fanofreerankformula} are zero for $m > \frac{p^e - 1}{r} +n$. For $e \geq 1$, let $D_e $ denote the effective divisor $\lceil (p^e -1) \Delta \rceil$. Therefore, by \Cref{dualIe}, we have
$$ \sum _{m = \lceil \frac{p^e -1}{2r}\rceil} ^{\lfloor \frac{p^e -1}{r} \rfloor} \dim_{k} \frac{H^{0}(mL)}{\Ied(mL)} = \sum _{m = \lceil \frac{p^e -1}{2r}\rceil} ^{\lfloor \frac{p^e -1}{r} \rfloor} \dim_{k} \frac{H^{0}((1-p^e) K_{X} - D_e - mL)}{\Ied((1- p^e)K_{X} - D_e -mL)}.$$
Next, for each $e \geq 1$, we may write $(1- p^e) K_X - D_e - mL = (\lfloor \frac{p^e -1}{r} \rfloor - m) L - D'_e$ for some Weil-divisor $D'_e$.
Moreover, we observe that as $e$ varies, the divisor $D'_e$ varies over only finitely many different Weil-divisors. This is clear by writing
\[ D'_e =  (p^e - 1 - r \lfloor \frac{p^e -1}{r} \rfloor )K_X + \lceil (p^e -1) \Delta \rceil  -  \lfloor \frac{p^e -1}{r} \rfloor r \Delta .\] 
Thus, using \Cref{twistedFsig}, we see that there is a constant $C>0 $ such that
\[ \left| \frac{H^{0}(nL- D'_e)}{\Ied(nL-D'_e)} - \frac{H^{0}(nL)}{\Ied(nL)} \right| < C p^{e (\dim(X) -1)} \]
for all $n \geq 1$ and all $e \geq 1$.

Thus, there exists a postive constant $C' > 0$ such that for all $e \geq 1$,
\[ \left | a_e ^\Delta - \sum _{m = 0} ^{\lfloor \frac{p^e -1}{2r} \rfloor} \dim_{k} \frac{H^{0}(mL)}{\Ied(mL)}  - \sum _{m = \lceil \frac{p^e -1}{2r} \rceil } ^{\lfloor \frac{p^e -1}{r} \rfloor} \dim_{k} \frac{H^{0}((\lfloor \frac{p^e -1}{r} \rfloor -m)L)}{\Ied((\lfloor \frac{p^e -1}{r} \rfloor -m)L)} \right | < C' p^{e\dim(X)}. \]

Finally, this implies that
\[ \left| a _e ^\Delta -  2 \sum \limits _{m = 0} ^{\lfloor \frac{p^e - 1}{2r}\rfloor}  \dim_{k} \frac{H^{0}(mL)}{\Ied(mL)} \right| < C'p^{e\dim (X)}. \]
The proof is now complete since the right hand side limits to zero when divided by $p^{e (\dim(X) +1)}$ and as $e \to \infty$. 
\end{proof}

\begin{proof}[\textbf{Proof of \Cref{alphavss}}]
     The proof of this Theorem is similar to the proof of \Cref{comparisonwithFsig} in \Cref{positivityofalpha} (see the proof of \Cref{postivitylemma}), once we replace the formula from  \Cref{formulaforFsigwithdelta} with the formula from \Cref{Formulafors} to compute the $F$-signature. Let $r$ be a positive integer such that $r\Delta$ is $\ZZ$-Weil and $L := r(K_X+ \Delta)$ is Cartier. Let us denote $S:= S(X, L)$, the section ring of $X$ with respect to $L$ and $ \tilde{\Delta}$, the cone over $\Delta$ with respect to $L$.  If we pick a homogeneous element $f \in S$ of degree $n$ such that $\fpt (S, \tilde{\Delta}; f) < \frac{1}{p^{e_0} -1} < \frac{\lambda}{n}$ for a real number $\lambda > 0$, then the argument of \Cref{postivitylemma} gives us that 
     \[
  \dim_{k} \frac{S_m}{I_{e_{0}}^\Delta (mL)} \leq \dim_{k} S_m - \dim_k   S_{m-n} \]
 for all $m \geq n$. Similarly, setting $v_{r} = \frac{p^{re_{0}} -1}{p^{e_0 } - 1} $ for any integer $r$, we have that $\fpt(S, \tilde{\Delta}; f^{v_r}) < \frac{1}{p^{re_0} -1}$, and so $f^{v_r} $ belongs to  $I_{re_{0}} ^{\Delta}(n v_{r} L)$. Therefore, we have
\begin{equation} 
  \dim_{k} \frac{S_m}{I_{ne_{0}} ^\Delta(mL)} \leq \dim_{k} S_m - \dim_{k} S_{m-nv_{r}}  
\end{equation}
for all $m \geq n v_{r}$. Then, using \Cref{Formulafors} we may compute the $F$-signature $\s(S, \tilde{\Delta})$ as the limit (denoting the dimension of $X$ by $d$):
$$ \s(X, \Delta) = \lim _{e \to \infty } \frac{ 2r \,  \sum \limits _{m = 0} ^{\lfloor \frac{p^e - 1}{2r}\rfloor}  \dim_{k} \frac{H^{0}(-mL}{\Ied(-mL)}}{ p^{e(d+ 1)}} \leq \Bigg( \lim _{r \to \infty } \frac{2r \,  \sum \limits _{m = 0} ^{\lfloor \frac{p^e - 1}{2r}\rfloor}  \dim_{k} S_m}{p^{re_{0} (d+1)}} - \frac{ 2r \,  \sum \limits _{m = nv_{r}} ^{\lfloor \frac{p^e - 1}{2r}\rfloor}  \dim_{k} S_{m- n v_{r}}}{p^{re_{0} (d+1)}} \Bigg).$$

Computing the dimensions using the asymptotic Riemann-Roch formula, we obtain
\[\s(X, \Delta) \leq \frac{2r \, \vol(L)}{ (2r)^{d+1} (d+1)!} \big( 1 - (1 - \frac{2rn}{p^{e_{0}} -1}) ^{d+1} \big).  \]
Note that $\vol(L) = r^d \vol((-K_X - \Delta)) $, and so the above formula can be simplified to
\begin{equation} \label{eqn:finalsinequality} \s(X, \Delta) \leq \frac{ \, \vol(-(K_X + \Delta))}{ (2^d \,  (d+1)!} \big( 1 - (1 - \frac{2rn}{p^{e_{0}} -1}) ^{d+1} \big) \leq \frac{ \, \vol(-(K_X + \Delta))}{ (2^d \,  (d+1)!} \big( 1 - (1 - 2r \lambda) ^{d+1} \big). \end{equation}
Now, we may pick a sequence of functions $f_t$ of degree $n_t >0$ such that
\[ \lim_{t \to \infty} (n_t \, \fpt(S, \tilde{\Delta}; f_t)) = \frac{1}{r} \, \FA(X, \Delta). \] 
By replacing $f_t$ by suitable powers if necessary, we may also assume that $\fpt(S, \tilde{\Delta}; f_t) \leq \frac{1}{p^{e_t -1}}$ for some $e_t >0$ such that $\lim_{t \to \infty} e_t = \infty$. Therefore, combining \Cref{eqn:finalsinequality} with the fact that $\lim_{t \to \infty} \frac{n_t}{p^{e_t -1}} = \frac{\FA(X, \Delta)}{r}$ we obtain the desired right inequality in \Cref{alphavsscor}.

The proof of the left inequality is exactly the same as in \Cref{positivityofalpha}, so we omit the details here.
\end{proof}

\begin{proof}[Proof of \Cref{alphavsscor}]
    Part (1) follows immediately from the right-hand inequality in \Cref{alphavss}, since we know that $\FA(X, \Delta) \leq \frac{1}{2}$ by \Cref{alphathm}. We also see that if $\FA(X, \Delta) < \frac{1}{2}$, we must have
    \[ \s(X, \Delta) < \frac{\vol(X, \Delta)}{2^{d} (d+1)!}. \] Thus, Part (2) also follows from \Cref{alphavssequation} once we note that when $\FA(X, \Delta) = \frac{1}{2}$, both sides of the inequality in \Cref{alphavssequation} are equal to $ \frac{\vol(X, \Delta)}{2^{d} (d+1)!} $.
    \end{proof}

\begin{rem}
    A special case of Part (1) of \Cref{alphavsscor}, when $-K_X$ is Cartier (so that $S(X, -K_X)$ is a graded Gorenstein ring) was proved in \cite{SannaiWatanabeFsignatureofGorensteinrings}.
    Similarly, a local version was proved for $3$-dimensional Gorenstein rings in \cite[Theorem~4.15]{JNSWYboundsonhilbertkunzandFsignature}. In this spirit, part (1) of \Cref{alphavsscor} can be viewed as a vast generalization of these results to section rings of $\QQ$-Fano varieties and log Fano pairs. Moreover, in practice, estimates on the $\FA$-invariant provide a natural strengthening of the bounds from \cite{SannaiWatanabeFsignatureofGorensteinrings} and \cite{JNSWYboundsonhilbertkunzandFsignature} via \Cref{alphavss}. Furthermore, when combined with \Cref{Prop:Addingeffective} and \Cref{positivityofalpha}, we may obtain upper bounds on the $F$-signature for section rings that are not Gorenstein as well.
\end{rem}

\subsection{Analogy with the complex $\alpha$-invariant.}

Now we explain why the $\FA$-invariant can be considered to be a ``Frobenius-version" of the complex $\alpha$-invariant. We refer to \cite{BlumJonssonDeltaInvariant} for the details.

Let $X$ be a $\QQ$-Fano variety over $\CC$. This means that $X$ is a normal variety, $-K_X$ is $\QQ$-Cartier and ample and $X$ has only klt singularities.

\begin{dfn} (\cite[Theorem A.3]{CheltsovShramovDemaillyalphainvariant}, \cite{TianAlphaDefinition})  \label{complexalphadfn}
    The $\alpha$-invariant of $X$, denoted by $\alpha (X)$ is defined as
\[ \alpha (X) := \inf \{ \, \text{lct}(X, D) \, | \, \text{$D$ - effective $\QQ$-divisor on $X$ with } D \sim_\QQ -K_X \}. \]
\end{dfn}
Here, $\text{lct}(X,D)$ is the \emph{log canonical threshold}, a numeric invariant that measures the log canonicity (or log terminality) of the pair $(X, D)$. This invariant is closely related to the $F$-pure threshold (\Cref{Fptdfn}) considered in this paper in the following sense: It follows from \cite{HaraYoshidaGeneralizationOfTightClosure} that if $(R, \fm)$ is a normal, $\QQ$-Gorenstein local ring, essentially of finite type over $\CC$, and if $(Y = \Spec(R), D)$ is log-pair (i.e., $K_Y + D$ is $\QQ$-Cartier), then we have
\begin{equation} \label{fptapproacheslct} \lim_{p \to \infty} \text{fpt}_{\fm_p}(Y_p, D_p) = \text{lct}_\fm (Y,D) \end{equation}
where $(Y_p, D_p)$ denotes the reduction to characteristic $p \gg 0$ (via a model over a finitely generated $\ZZ$-algebra) of $(Y,D)$. In this sense, the $F$-pure threshold is a Frobenius-analog of the log canonical threshold. Moreover, the $F$-pure threshold satisfies many of the formal properties of the log canonical threshold.

However, while the log canonical threshold in the definition of the $\alpha_\CC$-invariant is computed on $X$, the definition of the $\FA$-invariant in \Cref{FASdfn} involves the $F$-pure threshold on the cone over $X$. This is justified as follows:

\begin{rem} \label{comparisontoCremark}
     Let $r$ be such that $-rK_X$ is Cartier, and $S = S(X, -rK_X)$ be the corresponding section ring. Then, for any effective $\QQ$-divisor $D$ on $X$ with $D \sim _\QQ  -K_X$, we have
    \[  \text{lct}_\fm (S, D_S) = \min \, \{ \, \text{lct}(X, D), \, \frac{1}{r} \,  \}  \]
    where lct denotes the log canonical threshold and $\Delta_S$ denotes the cone over $\Delta$. This follows from \cite[Lemma 3.1]{KollarKovacsSingularitiesBook}. Thus, if we let 
    \[ \tilde{\alpha} (X) = r \, \inf \,  \{ \text{lct}_\fm (S , \Delta_S) \, | \, \Delta \geq 0 \, \text{is a $\QQ$-divisor on $X$ such that $\Delta \sim _\QQ -K_X$} \, \}, \]
    then, we have that
    \[  \tilde{\alpha}(X) = \min \, \{ \alpha(X), \, 1 \}.\]
\end{rem}
Therefore, for any $\QQ$-Fano variety $X$, the $\FA$-invariant (\Cref{FAXdfn}) can be considered to be the Frobenius-version of the quantity $  \min \, \{ \alpha(X), \, 1 \}$. Thus, at least when $\alpha(X) \leq 1$, we may expect that the $\FA$-invariant and the complex $\alpha$-invariant are related by the reduction modulo $p$ process thanks to \Cref{fptapproacheslct}. However, \Cref{alphathm} tells us that this doesn't have to be the case.

\begin{rem}
    The complex $\alpha$-invariant of the cubic surface defined by $x^3 + y^3 + z^3 + w^3$ is equal to $2/3$ (see \cite[Theorem 1.7]{CheltsovAlphaofDelPezzo}). But by \Cref{alphathm}, in every characteristic $p \geq 5$, we have $\FA(X_p) \leq 1/2$, where $X_p$ denotes the same cubic surface now considered over $\FF_p$ (or $\overline{\FF}_p$). This example points to the limitations of approximating the log canonical threshold by $F$-pure threshold for an unbounded family of divisors on $X$. Also see \Cref{cubicsurfaceeg} for more on the $\FA$-invariant of this cubic surface.
\end{rem}

\subsection{The $\FA$-invariant of toric Fano varieties}

Now we prove that for toric $\QQ$-Fano varieties, the $\FA$-invariant (in any characteristic) is the same as the complex $\alpha$-invariant.

Let $k$ denote an algebraically closed field of prime characteristic $p> 0$. Fix a lattice $N \isom \ZZ^d$ and let $M$ be the dual lattice (where $d$ is some positive integer).

\begin{thm} \label{Toricalphathm}
    Let $X_{p}$ be a $\QQ$-Fano toric variety over $k$ defined by a fan $\cF$ in $N$. Let $X_{\CC}$ be the corresponding complex toric variety (which is also automatically $\QQ$-Fano). Then, we have
    $$ \FA(X_{p}) = \alpha (X_{\CC}). $$ 
\end{thm}

\begin{proof}
Let $v_1, \dots, v_n$ denote the primitive generators for the one dimensional cones in $\cF$ and write $-K_X = \sum_i b_i v_i$ for rational numbers $b_i$.

    First we choose an $r > 0$ such that $rb_i \in \ZZ$ for each $i$ and the section ring $S(X, -rK_X)$ is generated in degree one. Let $P \subset M $ denote the polytope associated to $-rK_X$, and  defined by:
    \[ P = \{ u \in M_\RR = M \otimes_\ZZ \RR \, | \, \langle u, v_i \rangle \geq -b_i \, \, \text{for all } 1 \leq i \leq n \}. \]
    Since we are assuming that $S(X, -rK_X)$ is generated in degree $1$, the vertices of $P$ are lattice points of $M$. For any $u \in P \cap M$, let $D_u$ be the corresponding effective divisor in the linear system $|-rK_X|$. By \cite[Corollary 7.16]{BlumJonssonDeltaInvariant}, we have that
    \begin{equation} \label{alphahomogeneouslct} \alpha(X_\CC) =  \min_{u \in P \cap M} \, r\, \text{lct}(X_\CC, D_u) \end{equation}
    where $\text{lct}(X_\CC, -)$ denotes the log canonical threshold of a divisor on $X_\CC$. Note that since the vertices of $P$ are lattice points, just the vertices are sufficient to compute $\alpha (X_\CC)$.
    
    Let $\tilde{P}$ denote the polytope $P \times \{1\} \subset M \times \ZZ $. Then, the section ring $S(X, -rK_X)$ is the semigroup ring associated to the cone over $\tilde{P}$ in $(M \times \ZZ) \otimes_\ZZ \RR$. Note that $S$ is $\QQ$-Gorenstein. Therefore, by \cite[Theorem 3]{BlickleMultiplierIdealsAndModulesOnToric}, we see that for any $\tilde{u} \in \tilde{P} \cap (M \times \ZZ)$, we have 
    \begin{equation} \label{fpt=lct} \fpt(S(X_p, - K_{X_p}), D_{\tilde{u}}) = \text{lct}_\fm (S(X_\CC, -K_{X_\CC}), D_{\tilde{u}}). \end{equation}

    Next, note that since $X$ is a normal toric variety, it is automatically globally $F$-regular. Now we prove that the $\FA$-invariant of $X_p$ can also be computed by only considering the torus invariant divisors. To see this, let $S = S(X_p, -rK_{X_p})$ and let $f \in S $ be a non-zero homogeneous element. Then, following the discussion in  \cite[Section 7.4]{BlumJonssonDeltaInvariant} and \cite[Theorem 15.17]{EisenbudCommutativeAlgebraWithAView}, there exists an integral weight vector $\mu = (\mu_1, \dots, \mu_{d+1})$ with $\mu_i \in \ZZ_{>0}$ such that $\text{in}_{>_\mu} (f) = \text{in}_{>} (f)$. Here $>_\mu$ denotes the weight monomial order with respect to $\mu$ and $>$ denotes the graded lexicographic monomial order on $S$. Then,  we have a flat degeneration of $f$ to its initial term. In other words, if $f = \sum _u \beta_u \chi^u$ for monomials $\chi^u \in S$, then setting $w = \max \{\langle \mu ,  u \rangle\, | \, \beta_u \neq 0 \} $, the element 
    \[\tilde{f} = t^w \, \sum _u \beta_u t^{ -\langle \mu, 
    u \rangle} \chi^u \in S[t] \]
    satisfies the following properties:
    \begin{itemize}
        \item Viewing $ S[t]$ as a $k[t]$-algebra, the ring $  S[t]/ (\tilde{f})$ is a flat $k[t]$-module.
    \item The image of $\tilde{f}$ modulo $t$ is equal to $\text{in}_{>}(f)$, the initial term of $f$ with respect to the graded lex monomial order on $S$.
    \item   For any point $0 \neq \lambda \in k$, the image  $f_\lambda $ of $\tilde{f}$ in $S[t]/(t - \lambda)$ satisfies
        \[ \fpt(S, f_{\lambda}) = \fpt(S, f).  \]
    \end{itemize}
  With this construction in place, we conclude the proof of the theorem with the following lemma:
  \begin{lem} \label{fptofinitialtermlemma}
      For any non-zero homogeneous element $f$ of $S$, we have 
      \[ \fpt(S, f) \geq \fpt (S, \text{in}_> (f)) . \]
  \end{lem}

  Assuming this lemma for a moment, we see that
  \[ \FA(X_p) = \inf _{\tilde{u} \in \tilde{P} \cap (M \times \ZZ)}  \fpt (S, D_{\tilde{u}}). \]
Furthermore, by \Cref{fpt=lct}, we have
\begin{equation} \label{toricalphalct} \FA(X_p) = \inf _{\tilde{u} \in \tilde{P} \cap (M \times \ZZ)}  \text{lct}_\fm (S(X_\CC, -rK_{X_\CC}), D_{\tilde{u}}). \end{equation}
Since by \Cref{alphathm} we have $\FA(X_p) \leq 1/2$, we must have $\text{lct} (S(X_\CC, -rK_{X_\CC}), D_{\tilde{u}})  < \frac{1}{nr} $ for some $\tilde{u} =  (u, n) \in \tilde{P} \cap (M \times \ZZ)$. Note that $D_u$ corresponds to a torus-invariant divisor on $X_\CC$ linearly equivalent to $-nr K_{X_\CC}$.  Therefore, by \Cref{comparisontoCremark}, we have
\[ \text{lct} (X_\CC, D_u)  = \text{lct}_\fm (S(X_\CC, -rK_{X_\CC}), D_{\tilde{u}}) \]
for any $\tilde{u} = (u, n)$ such that $\text{lct} (S(X_\CC, -rK_{X_\CC}), D_{\tilde{u}})  < \frac{1}{nr}$.
Putting this together with \Cref{toricalphalct} and \Cref{alphahomogeneouslct}, we get that 
\[  \FA(X_p) = \alpha(X_\CC) \]
as required.
\end{proof}

Finally, it remains to prove \Cref{fptofinitialtermlemma}.

\begin{proof}[\textbf{Proof of \Cref{fptofinitialtermlemma}}]
     By \Cref{fptvssplitting}, it is sufficient to show that for all rational numbers of the form $\frac{a}{p^e -1}$ such that 
     \[ \frac{a}{p^e -1} < \fpt (S, \text{in}_> (f)), \] the map 
     $S \to \Fe S$ sending $1$ to $ \Fe f^a$
     splits. Equivalently, for all such $\frac{a}{p^e -1}$, it suffices to show that $f^a  \notin I_e (S)$. Since the pair $(S,  \text{in}_> (f)^{\frac{a}{p^e -1}})$ is strongly $F$-regular, in particular it is sharply $F$-split. By \Cref{fptvssplitting} again, we know that $\text{in}_> (f^a) = (\text{in}_> (f))^a \notin I_e(S)$. Now, since $S$ is a toric ring, $I_e (S)$ is a monomial ideal of $S$. Therefore, if $\text{in}_> (f^a) \notin I_e(S)$, we also have $f^a \notin I_e(S)$ as required.
\end{proof}

\begin{rem}
    Combining \Cref{Toricalphathm} with \Cref{alphathm}, we recover the well-known fact that the $\alpha$-invariant of a toric Fano variety is at most 1/2 (see \cite[Corollary~3.6]{LiuZhuangSharpnessofTianscriterion}). 
\end{rem}

\section{Examples}

In this section, we compute some examples of the $\FA$-invariant for non-toric varieties and highlight some interesting features.

\subsection{Quadric hypersurfaces}

Fix any algebraically closed field $k$ of characteristic $p \neq 2$ and let $Q_d \subset \PP^{d+1}$ be the $d$-dimensional smooth quadric hypersurface over $k$. Note that by the adjunction formula, $-K_{Q_d}=  d H$ where $H$ denotes a hyperplane section.

\begin{eg} \label{quadriceg}
    Then, $\FA (Q_d) = \frac{1}{d}$. Equivalently, if $S$ denotes the section ring 
    \[ S:= S(Q_d, \cO_{Q_d} (1)) \isom k[x_0, \dots, x_{d+1}]/(x_0 ^2 + \dots x_{d+1} ^2), \]
    then $\FA(S) = \FA(Q_d, \cO_{Q_d}(1)) = 1$. This follows from a description of the structure of the sheaves $\Fe ( \cO_{Q_d} (m))$ proved in \cite{LangerFrobeniusPushforwardsonQuadrics} and \cite{AchingerFrobeniuspushforwardonQuadrics}. More precisely, for any $e \geq 1$ and $0 \leq m \leq p^e -1$, \cite[Theorem 2]{AchingerFrobeniuspushforwardonQuadrics} tells us that $\Fe(\cO_{Q_d} (m))$ is a direct sum of $\cO_{Q_d} (-t)$ and $\cS (-t)$ for $t \geq 0$, where $\cS$ is an ACM bundle that sits in an exact sequence of the form
    \[ 0 \to \cO_{\PP^{d+1}} (-2) ^{\oplus a} \to \cO_{\PP^{d+1}} (-1)  ^{\oplus b} \to i_* \cS \to 0 \]
    for suitable positive integers $a$ and $b$. Here $i: Q_d \hookrightarrow \PP^{d+1}$ is the inclusion. See \cite[Section 1.3]{AchingerFrobeniuspushforwardonQuadrics} for the details. Since $H^1 (\PP^{d+1}, \cO_{\PP^{d+1}}(-2)) = 0$, we deduce from the exact sequence above that $\cS(-t)$ has no global sections for any  $t \geq 0$. Therefore, all global sections of $\Fe(\cO_{Q_d} (m))$ appear in the trivial summands. In other words, $I_e (\cO_{Q_d}(m)) = 0$ for any $e \geq 1$ and any $m \leq p^e -1$. Moreover, since $0 \neq x_1 ^{p^e}\in I_e (S)$, we know that $m_e = p^e -1$. Therefore, by \Cref{finitedegreeapprox}, we have
    \[  \FA(S) = \FA(Q_d , \cO_{Q_d}(1) = \lim _{e \to \infty} \frac{p^e -1 }{p^e} = 1.\]
    We also note that as a consequence, $X = \PP^1 \times \PP^1$ is a del Pezzo surface with $\FA(X) = 1/2$ (since $-K_X = \cO_X(2)$). Thus, by \Cref{alphavsscor}, we immediately obtain that $\s(X) = \frac{1}{3}$. Of course, since $X$ is a toric surface, this $F$-signature can be computed in various different ways as well.
\end{eg}

\begin{rem}
    This example shows that the $\FA$-invariant does not characterize regularity of section rings, since the $\FA$-invariant of a polynomial ring is also equal to $1$.
\end{rem}

\begin{rem}
    Another interesting feature of this example is that the $\FA$-invariant of smooth quadrics is independent of the characteristic $p$ (for $p \neq 2$). This is far from true in general. Furthermore, for any $d > 2$, the $F$-signature of $Q_d$ is known to depend on $p$ in a rather complicated way (see \cite{TrivediHKDensitfunctionofquadrics}).
\end{rem}

\subsection{Full flag varieties.}
\begin{eg} \label{eg:fullflagvars} Let $k$ be an algebraically closed field of characteristic $p>0$ and let $X_n$ denote the full flag variety parametrizing complete flags in a fixed $n$-dimensional vector space $V$ over $k$. Recall that for each $n \geq 2$, $X_n$ is a smooth Fano variety of dimension $d := \frac{n(n-1)}{2}$. Moreover, $X_n$ is also known to be globally $F$-regular in every positive characteristic $p >0$ (see \cite{LauritzenRabenThomsenGlobalFRegularityOfSchuertVarieties}). We have that $\FA(X_n) = \frac{1}{2} $ for $n \geq 2$, and consequently by \Cref{alphavsscor}, we get that
\[ \s(X_n) = \frac{\vol(-K_{X_n})}{2^d (d+1)!} = \frac{1}{d+1}.\]

Recall that the Picard group of $X$ can be identified with the quotient $\ZZ^{n}/\langle (1, \dots, 1) \rangle$ such that the ample line bundles correspond to the positive dominant roots, which is defined as the set
\[ \text{Amp}_{X_n} = \{(\lambda_1, \dots, \lambda_n) \,| \, \lambda_{1} > \lambda_2> \dots >\lambda_n \}. \]
In this set, the anti-canonical line bundle $\omega_{X_n} ^{-1}$ corresponds to the root $(2n, 2n -2, \dots, 2)$. Let $L$ denote the line bundle corresponding to the root $(n, n-1, \dots, 1)$. Then, the Hilbert polynomial of $L$ is given by $H(m) = (m+1) ^d$ (see \cite[Section 2.1]{FakhruddinTrivediHilbertKunzMultiplicity}). Together with the Kempf vanishing theorem (which follows from global $F$-regularity and \Cref{vanishingforGFR}), we conclude that for each $m \geq 1$, we have $\dim_k (H^0 (X_n, L^m)) = (m+1) ^d \neq 0$. In particular, we have that $\vol (L) = d!$ and $\vol(-K_{X_n}) = 2^d \, d!. $

Now, to justify that $\FA(X_n) = 1/2$, by \Cref{alternatecharofalpha}, it is enough to show that for each $e \geq 1$, and $m \leq p^e -1$, we have $I_e (mL) = 0$ (where $I_e$ is the subspace defined in \Cref{Iedfnwithdelta}). Using \Cref{lem:addingeffective} and the fact that we can choose an effective divisor corresponding to each $mL$ (with $m>0$), it is enough to see that $I_e ((p^e -1) L) = 0$. But this follows from the definition of $I_e$ and the following result proved by Haboush (\cite[Corollary~2.2]{HaboushAShortProofOfKempf}): for each $ e \geq 1$, the sheaf $\Fe (L^{p^e -1})$ is free. In other words, we have an isomorphism
\[ \Fe (L^{p^e -1}) \isom \cO_{X_n} ^{p^{ed}}. \]
This key result is a consequence of the irreducibility of the Steinberg representations.

The value of the $F$-signature now follows immediately from \Cref{alphavsscor} and the computation that $\vol(-K_{X_n}) = 2^d \, d!$.
\end{eg}

\begin{rem}
    Using similar inputs from representation theory, the Hilbert-Kunz multiplicity of full flag varieites was computed by Fakhruddin and Trivedi in \cite{FakhruddinTrivediHilbertKunzMultiplicity}.
\end{rem}

\begin{rem}
    We thank Vijaylaxmi Trivedi for pointing us to Haboush's result and Claudiu Raicu for valuable conversations about this example.
\end{rem}

\subsection{Diagonal cubic surface}
Let $k= \FF_p$ for some prime number $p \geq 5$ and $X_p \subset \PP^3$ be the diagonal cubic surface defined by $x^3 + y^3 + z^3 + w^3 = 0$ over $k$.

\begin{eg} \label{cubicsurfaceeg}
    For each $p \geq 5$, we have $\FA(X_p) < \frac{1}{2}$. However,
    \[ \lim _{p \to \infty} \FA(X_p) = \frac{1}{2}. \]
    To see this, we using the following result proved by Caminata, Shideler, Tucker and Zerman in \cite[Theorem~7.1]{CaminataShidelerTuckerZermanFsignatureofdiagonalhypersurfaces} (see also \cite[Section 5]{ShidelerUGthesis}), building on the techniques of Han and Monsky: Let $\s_p$ denote the $F$-signature of $X_p$, equivalently, of the ring $\FF_p [x,y,z,w]/(x^3 + y^3 + z^3 + w^3)$. Then for any $p \geq 5$, writing $p = 3n+1 $ or $p = 3n+2$, we have
    \begin{equation} \label{equation:Fsigcubicsurface} \s_p = \frac{1}{8} - \frac{1}{4 (27n^2 + 27 n +8)}  . \end{equation}
    In particular, for all $p \geq 5$, we have $  \s_p < \frac{1}{8}$ .
    Moreover,
    \[ \lim _{p \to \infty} \s_p = \frac{1}{8}.\]
    Using this, our claims about the $\FA$-invariant of $X_p$ follow from \Cref{alphavss} and \Cref{alphavsscor}, once we observe that
    \[ \frac{\vol(-K_{X_p})}{2^2 \, 3!} = \frac{1}{8}. \]
\end{eg}

\begin{eg} \label{cubicsurfacecontd}
    Continuing with the previous example, let $S_p$ denote the polynomial ring $\FF_p [x,y,z,w]$ and $G$ denote the polynomial $x^3 + y^3 + z^3 + w^3$, and set $R_p = S_p / (G)$.
    Then, we may use the theory of the $\FA$-invariant to directly verify (at least for small primes $p$) that $\s_p < 1/8$. This can be done by showing that the ideal $I_1 = ((x^p, y^p, z^p, w^p) :_{S_p} (G^{p-1}))R$ contains a non-zero element of degree equal to $\frac{p-1}{2}$ (see \Cref{eg:hypersurfaceIe}). Indeed, by Part (2) of \Cref{alternatecharofalpha}, this implies that $\FA(X_p) < 1/2$, which in turn implies that $\s_p < 1/8$ by \Cref{alphavsscor}. For small values of $p$, we may compute this ideal using Macaulay2 (\cite{M2}) to see that
    \begin{enumerate}
        \item for $p = 5$, we have $x^2 \in I_1$,
        \item for $p =7$, we have $xyz \in I_1$,
        \item for $p = 11$, we have $x^2 (z^3 - w^3) \in I_1$.
        \item for $p = 31$, we have $xyw(z^{12} - 10 z^9 w^3+ 15z^6 w^6 - 4 z^3 w^9 + 12 w^{12}) \in I_1$.
    \end{enumerate}
    From these calculations, it seems reasonable to predict that $I_1$ always contains a non-zero element of degree $\frac{p-1}{2}$. But there does not seem to be a clear pattern in these elements as $p$ varies.
\begin{question}
    Does there exist a non-zero element of degree $\frac{p-1}{2}$ in $I_1$ of $R_p$ for each $p \geq 5$? Is there a geometric description of such elements (in terms of the divisors on the cubic surface)?
\end{question}

    In the other direction, using the bounds in \Cref{alphavss} and the formula in \Cref{equation:Fsigcubicsurface} we see that for \emph{each} $p \geq 5$, we have $\FA(X_p) \geq \frac{1}{3}$. This implies (by \Cref{alternatecharofalpha}) that the ideal $I_1$ does not contain any non-zero elements of degree less than $\frac{p-1}{3}$. We do not know any direct way to prove this without relying on the calculations from \cite{CaminataShidelerTuckerZermanFsignatureofdiagonalhypersurfaces} and the theory of the $\FA$-invariant.
    \end{eg}

\subsection{Further questions}
By analogy with the complex $\alpha$-invariant, we hope that the $\FA$-invariant may have extensive applications in the study of Fano varieties in positive characteristics. In this direction, we end by stating two important questions.

\begin{question}
    Is the infimum in the definition of the $\FA$-invariant (\Cref{FASdfn}) a minimum?
\end{question}

In the case of the complex $\alpha$-invariant, this question has an affirmative answer under the assumption that $\alpha(X) \leq 1$ thanks to the deep results of Birkar proving the BAB-conjecture (see \cite[Theorem 1.7]{BirkarBABConjecture}).

\begin{question}
    Suppose $V$ and $\alpha $ are fixed positive real numbers and $d$ be a fixed positive integer. Then, does the set of globally $F$-regular $\QQ$-Fano varieties satisfying $\FA(X) \geq \alpha$ and $\vol(-K_X) \geq V$ form a bounded family?
\end{question}

    This question is significant from the viewpoint of moduli theory of Fano varieties. The analog for complex Fano varieties was proved by Jiang in \cite{Jiangboundednesviaalphainvariant}, again relying on the resolution of the BAB conjecture by Birkar.

 \bibliographystyle{alpha}
    \bibliography{Main}
\end{document}